\documentclass[a4paper,leqno,12pt,twoside]{amsart}
\usepackage{amsmath,amsfonts,amssymb,amsthm,amscd}
\usepackage[colorlinks=true,citecolor=blue, urlcolor=blue, linkcolor=blue,pagebackref]{hyperref}
\usepackage[top=1.2in,bottom=1.2in,left=1.in,right=1.in]{geometry} 
\usepackage[dvipsnames]{xcolor}
\usepackage{tikz}
\usepackage{amsthm, amsmath, amssymb,latexsym}
\usepackage{tikz-cd}
\usepackage[all]{xy}

\numberwithin{equation}{section}

\newtheorem{theorem}{Theorem}[section]
\newtheorem{lemma}[theorem]{Lemma}
\newtheorem{corollary}[theorem]{Corollary}
\newtheorem{proposition}[theorem]{Proposition}

\theoremstyle{definition}
\newtheorem{remark}[theorem]{Remark}
\newtheorem{definition}[theorem]{Definition}

\newcommand{\nc}{\newcommand} 

\nc{\cH}{{\mathcal H}}
\nc{\cA}{{\mathcal A}}
\nc{\cG}{{\mathcal G}}
\nc{\cC}{{\mathcal C}}
\nc{\cO}{{\mathcal O}}
\nc{\cI}{{\mathcal I}}
\nc{\cB}{{\mathcal B}}
\nc{\cY}{{\mathcal Y}}
\nc{\cK}{{\mathcal K}} 
\nc{\cX}{{\mathcal X}}
\nc{\cS}{{\mathcal S}}
\nc{\cE}{{\mathcal E}}
\nc{\cF}{{\mathcal F}}
\nc{\cZ}{{\mathcal Z}}
\nc{\cQ}{{\mathcal Q}}
\nc{\cN}{{\mathcal N}}
\nc{\cP}{{\mathcal P}}
\nc{\cL}{{\mathcal L}}
\nc{\cM}{{\mathcal M}}
\nc{\cT}{{\mathcal T}}
\nc{\cW}{{\mathcal W}}
\nc{\cJ}{{\mathcal J}}
\nc{\cV}{{\mathcal V}}
\nc{\bH}{{\mathbb H}}
\nc{\bA}{{\mathbb A}}
\nc{\bG}{{\mathbb G}}
\nc{\bC}{{\mathbb C}}
\nc{\bO}{{\mathbb O}}
\nc{\bI}{{\mathbb I}}
\nc{\bB}{{\mathbb B}}
\nc{\bY}{{\mathbb Y}}
\nc{\bK}{{\mathbb K}} 
\nc{\bX}{{\mathbb X}}
\nc{\bS}{{\mathbb S}}
\nc{\bE}{{\mathbb E}}
\nc{\bF}{{\mathbb F}}
\nc{\bZ}{{\mathbb Z}}
\nc{\bQ}{{\mathbb Q}}
\nc{\bN}{{\mathbb N}}
\nc{\bP}{{\mathbb P}}
\nc{\bL}{{\mathbb L}}
\nc{\bM}{{\mathbb M}}
\nc{\bT}{{\mathbb T}}
\nc{\bW}{{\mathbb W}}
\nc{\bU}{{\mathbb U}}
\nc{\bD}{{\mathbb D}}
\nc{\bJ}{{\mathbb J}}
\nc{\bV}{{\mathbb V}}
\nc{\bbZ}{{\mathbb Z}}
\nc{\bR}{{\mathbb R}}
\nc{\co}{{\nabla}}
\newcommand{\OO}{\mathcal{O}}
\newcommand{\la}{\longrightarrow}
\newcommand{\im}{\operatorname{im}}
\newcommand{\debar}{{\overline{\partial}}}
\newcommand{\HG}{\operatorname{HG}}
\newcommand{\reg}{{\operatorname{reg}}}
\newcommand{\s}[1]{s_{#1}}
\newcommand{\suz}{s_{z,\nu,x}}

\begin{document}

\title[A bidifferential and second fundamental form of Prym map]{A twisted bidifferential,
second fundamental form of Prym map and projective structures}

\author[I. Biswas]{Indranil Biswas}

\address{Department of Mathematics, Shiv Nadar University, NH91, Tehsil
Dadri, Greater Noida, Uttar Pradesh 201314, India}

\email{indranil.biswas@snu.edu.in, indranil29@gmail.com}

\author[E. Colombo]{Elisabetta Colombo}

\address{Universit\`a degli Studi di Milano, Dipartimento di Matematica, Via Cesare Saldini 50,
20133 Milano, Italy}

\email{elisabetta.colombo@unimi.it}

\author[P. Frediani]{Paola Frediani}

\address{Universit\`a degli Studi di Pavia, Dipartimento di Matematica, Via Ferrata 1, 27100 Pavia, Italy}
 \email{paola.frediani@unipv.it}

 \author[A. Ghigi]{Alessandro Ghigi}
 \address{Universit\`a degli Studi di Pavia, Dipartimento di Matematica, Via Ferrata 1, 27100 Pavia, Italy}
 \email{alessandro.ghigi@unipv.it}



\begin{abstract}
Given a point $[(\widetilde{C},C)]$ of the moduli space ${\mathcal R}_g$, where
$\widetilde{C}$ is an \'etale double cover of a smooth curve $C$ of genus $g \,\geq\, 7$, we
construct a twisted
meromorphic bidifferential $\widehat{\eta}_P$ on $C\times C$ with a pole on the diagonal. It is shown that
the second fundamental form $\rho_P$ of the Prym map from ${\mathcal R}_g$
to ${\mathcal A}_{g-1}$ at $[(\widetilde{C},C)]$ is given by this $\widehat{\eta}_P$. This allows us to compare
$\rho_P$ with the second fundamental forms of the two Torelli maps, to ${\mathcal A}_{2g-1}$ and
${\mathcal A}_{g}$, at the points $[\widetilde{C}]$ and $[C]$ respectively.
It is shown that $\widehat{\eta}_P$
induces a projective structure on $C$ that depends on the double cover $\widetilde{C}\, \longrightarrow\, C$;
these projective structures together give a 
${\mathcal C}^{\infty}$ section $\beta_P$ of the bundle of projective structures on
the moduli space ${\mathcal R}_g$. Finally, it is 
proved that the $(0,\,1)$-component of the differential of the section $\beta_P$ is the pullback of the K\"ahler 
form of the Siegel metric on ${\mathcal A}_{g-1}$ via the Prym map. This is a key ingredient in
the proof that this section 
$\beta_P$ is not the pullback of a section of projective structures on ${\mathcal M}_{g}$ and, in addition, 
that it differs from the known projective structures on ${\mathcal R}_g$.
\end{abstract}

\thanks{I. Biswas is partially supported by a J. C. Bose Fellowship
(JBR/2023/000003). E. Colombo, P. Frediani and A. Ghigi are members
of GNSAGA (INdAM) and are partially supported by PRIN project {\em
Moduli spaces and special varieties} (2022). }

\maketitle

\tableofcontents

\section{Introduction}

Denote by ${\mathcal R}_g$ the moduli space of pairs $(\widetilde{C},\,C)$, where $\pi\,:\,
\widetilde{C}\, \longrightarrow\, C$ is an
\'etale double cover of a smooth complex projective curve $C$ of genus $g$. So
${\mathcal R}_g$ is also the moduli space of pairs $[(C,\, \alpha)]$, where $C$ is as above and
$\alpha$ is a nontrivial holomorphic line bundle over $C$ equipped with a holomorphic trivialization of
$\alpha^{\otimes 2}$. This ${\mathcal R}_g$ has the Prym map to the moduli space of principally polarized
abelian varieties ${\mathcal A}_{g-1}$ of dimension $g-1$.

Given a point $[(\widetilde{C},C)]$ of ${\mathcal R}_g$, we construct a 
twisted meromorphic bidifferential $\widehat{\eta}_P$ on $C\times C$ with pole of order two
on the diagonal. This $\widehat{\eta}_P$ has two major properties that will be elucidated.

On the one hand, the second fundamental form $\rho_P$ of the Prym map 
at a point $[(\widetilde{C},C)]$ is given by multiplication by $\widehat{\eta}_P$.
This allows us to compare $\rho_P$ with the second fundamental
forms of the two Torelli maps, that sends a curve to its Jacobian, at the points $[\widetilde{C}]$ and $[C]$.
On the other hand, it is shown here that $\widehat{\eta}_P$ induces a projective structure both on
$C$ and on $\widetilde C$ that depends on the double cover $\pi\,:\, \widetilde{C}
\, \longrightarrow\, C$. This produces a ${\mathcal C}^{\infty}$ section $\beta_P$ of 
the holomorphic fiber bundle of projective structures on ${\mathcal R}_g$.
We compute $\overline{\partial}(\beta_P)$. These two properties of $\widehat{\eta}_P$
are explained below with more details.

The Prym variety $P(C,\, \alpha)$ is the connected
component --- containing the origin --- of the kernel of the norm map $Nm_{\pi}\,: \,
{\rm Pic}^0(\widetilde{C})\, \longrightarrow\, {\rm Pic}^0(C)$ that sends
any line bundle ${\mathcal O}_{\widetilde{C}}(\sum_i a_i p_i)$, with $\sum_i a_i\,=\, 0$,
to ${\mathcal O}_C(\sum_i a_i \pi(p_i))$, where $a_i\, \in\,
{\mathbb Z}$ and $p_i\, \in\, \widetilde{C}$. This $P(C, \alpha)$ is a principally polarized
abelian variety of dimension $g-1$; denote by $\Xi$ the principal polarization
on $P(C, \alpha)$. Let ${\mathcal A}_{g-1}$ be
the moduli space of principally polarized complex abelian varieties of dimension $g-1$. The Prym map 
$$
 {\mathcal P}\ :\ {\mathcal R}_g\ \longrightarrow\ {\mathcal A}_{g-1},\ \ [(C,\,\alpha)]\
\longmapsto\ [(P(C, \alpha), \Xi)],
$$ 
is generically an embedding for $g \,\geq\, 7$ \cite{friedman-smith},
\cite{kanev-global-Torelli}, and it is generically finite for
$g \,\geq\, 6$. Moreover, the map ${\mathcal P}$ is never injective and it has positive
dimensional fibers \cite{donagi}, \cite{mumford}, \cite{carlos}.

The Siegel space ${\mathcal H}_{g-1}\,=\, {\rm Sp}(2(g-1), {\mathbb R})/{\rm U}(g-1)$, which is a Hermitian 
symmetric domain, covers ${\mathcal A}_{g-1}$ in the sense that ${\mathcal A}_{g-1}\,=\, {\rm Sp}(2(g-1), 
{\mathbb Z})\backslash {\mathcal H}_{g-1}$. So the canonical symmetric metric on ${\mathcal H}_{g-1}$ descends 
to an orbifold metric on the quotient ${\mathcal A}_{g-1}$; this descended K\"ahler metric on ${\mathcal 
A}_{g-1}$ is also called the Siegel metric. Assume that $g \,\geq\, 7$, and denote by ${\mathcal R}^0_g$ the 
Zariski open subset of $ {\mathcal R}_g$ where the Prym map is an embedding. The (dual of) the second 
fundamental form of the Prym map with respect to this metric at a point $[(C,\, \alpha)] \,\in\, {\mathcal 
R}^0_g$ is a linear map
\begin{equation}\label{b0}
\rho_P\ :\ I_2(K_C \otimes \alpha)\ \longrightarrow\ {\rm Sym}^2 (H^0(C,\, K_C^{\otimes 2})),
\end{equation}
where $K_C$ is the canonical line bundle of $C$ and
$I_2(K_C \otimes \alpha)\, \subset\, {\rm Sym}^2 (H^0(C,\, K_C \otimes\alpha))$ is the kernel of the
natural multiplication map 
\begin{equation}
\label{mult}
m: {\rm Sym}^2 (H^0(C,\, K_C \otimes\alpha)) \ \longrightarrow\ H^0(C,\, K_C^{\otimes 2}).
\end{equation}
In \cite{cfadv} it was proved that $\rho_P$ in \eqref{b0} coincides with a
generalisation of the Hodge Gaussian maps given in \cite{cpt}. Here a different
and very explicit description of $\rho_P$ is given.

Set $S\,:=\, C\times C$, and denote by $\Delta\, \subset\, S$ the reduced diagonal divisor in the
surface $S$.
Let $p,\,q\,:\, S \,\longrightarrow\, C$ respectively be the projections to the first factor and the
second factor. The canonical line bundle $K_S$ of $S$ is isomorphic to $(p^* K_C) \otimes (q^*K_C)$.
Consider the holomorphic line bundle $\alpha_S\, := \,(p^*\alpha) \otimes (q^*\alpha)$ over $S$,
where, as before, $\alpha$ is a nontrivial 2-torsion point in ${\rm Pic}^0(C)$ equipped with a
trivialization of $\alpha^{\otimes 2}$. Then $I_2(K_C \otimes \alpha)\,=\, \text{kernel}(m)$
(see \eqref{mult}) can be identified with a subspace of $H^0(S,\, K_S\otimes \alpha_S(-2 \Delta))$. 

Our first main result is the following.

\begin{theorem}[{Theorem \ref{prodotto}}]\label{thmi1}
There is a certain explicit global section $\widehat{\eta}_P \,\in\, H^0(S,\, K_S\otimes 
\alpha_S(2 \Delta))$ that satisfies the following condition.
For $[(C,\, \alpha)]\, \in \,{\mathcal R}^0_g$, with $g\,\geq\, 7$,
the second fundamental form of the Prym map $$\rho_P\ :\ I_2(K_C \otimes \alpha)\ \longrightarrow\
{\rm Sym}^2(H^0(C,\,K_C^{\otimes 2}))$$ coincides with the restriction to $I_2(K_C \otimes \alpha)$ of the
multiplication map
$$
H^0(S,\, K_S \otimes \alpha_S(-2\Delta))\ \longrightarrow\ H^0(S,\,K_S^{\otimes 2} ),\ \ \ Q\
\longmapsto\ Q \cdot \widehat{\eta}_P .
$$
\end{theorem}

The second fundamental form of the Torelli map ${\mathcal M}_g\, \longrightarrow\,
{\mathcal A}_g$ at $C$ is given as a multiplication by a global section
$\eta\,\in\, H^0(S, \, K_S(2\Delta))$ \cite[Theorem 3.7]{cfg}. Theorem \ref{thmi1} is similar
in spirit but for the Prym map.

As a byproduct of Theorem \ref{thmi1} it follows immediately that the second fundamental form
$\rho_P$ of the Prym map is injective (see Corollary \ref{inj}).

Using the description of $\widehat{\eta}_P$, it is shown here --- with the appropriate 
identifications --- that the section $\widehat{\eta}_P$ can be realized as the difference of the sections 
$\widetilde{\eta}$ and $\eta$ that govern the second fundamental forms of the Torelli maps for 
$\widetilde{C}$ and $C$ respectively (see Theorem \ref{scomposizione}).

Our second main result shows that the section $\widehat{\eta}_P\,\in\, H^0(S,\,K_S \otimes \alpha_S(2 
\Delta))$ produces a projective structure on $C$.

\begin{theorem}[{Theorem \ref{betaP}}]\label{thmi2}
The restriction of $\widehat{\eta}_P\,\in \,H^0(S,\,K_S\otimes \alpha_S(2 \Delta))$ to the second infinitesimal
neighbourhood of the diagonal $\Delta$ in $C\times C$ (which is denoted as $3\Delta$) gives a section lying in $H^0(3\Delta,\,
K_{S}(2 \Delta)\big\vert_{3 \Delta})$ that induces a projective structure on the curve $C$ (this
projective structure on $C$ depends on $\alpha$). 
\end{theorem}

Using Theorem \ref{thmi2} for the moving point $[(C, \alpha)] \,\in\, {\mathcal R}_g$, we obtain a
${\mathcal C}^{\infty}$ section $\beta_P$ of the (affine) bundle $\widehat{\mathcal V}
\, \longrightarrow\, {\mathcal R}_g$ of projective
structures (see \eqref{vhat}). The $(0,\,1)$-component of the
differential of this section $\overline{\partial}(\beta_P)$ is a $(1,\,1)$-form on ${\mathcal R}_g$
(see Section \ref{se7}). The following theorem is proved:

\begin{theorem}[{Theorem \ref{debar}}]\label{thmi3}
The $(1,\,1)$-form $\overline{\partial}(\beta_P)$ on ${\mathcal R}_g$ is a nonzero scalar multiple of
${\mathcal P}^* \omega_S$, where $\omega_S$ is the K\"ahler form of the Siegel metric on ${\mathcal A}_{g-1}$,
and ${\mathcal P}$ is the Prym map. 
\end{theorem}

Theorem \ref{thmi3} may be compared with one of the main theorems of \cite{bcfp}. In fact, in \cite{bcfp} it 
is proved that the section $\eta \,\in\, H^0(S,\, K_S(2\Delta))$ --- that gives the 
second fundamental form of the Torelli map --- has the property that its restriction to $3 \Delta$ 
corresponds to a canonical projective structure on $C$. Consider the
${\mathcal C}^{\infty}$ section $\beta^{\eta}$ 
of the bundle of projective structures on the moduli space ${\mathcal M}_g$ obtained by varying 
the curve $C$ in ${\mathcal M}_g$. It is shown in \cite{bcfp} that
that $\overline{\partial}(\beta^{\eta})$ is a 
nonzero scalar multiple of the pullback of the K\"ahler form of the Siegel metric on ${\mathcal A}_g$ via the 
Torelli map (see \cite[Theorem 4.4 ]{bcfp}).

In \cite{bcfp} it was proved that the canonical section $\beta^{\eta}$ of projective structures on ${\mathcal 
M}_g$ induced by the second fundamental form of the Torelli map is actually different from the canonical section 
$\beta_u$ induced by the uniformisation theorem
of Riemann surfaces (see \cite[Theorem 4.1]{bcfp}). Another canonical ${\mathcal C}^{\infty}$ 
section of projective structures on ${\mathcal M}_g$ has been constructed in \cite{bgv1}, \cite{bv} using
the second order theta functions; these three sections are all different.

In Theorem \ref{different} it is shown that the section $\beta_P$ cannot be a pullback,
by the natural map ${\mathcal R}_g\,\longrightarrow\, {\mathcal M}_g$, of a 
section of projective structures on ${\mathcal M}_g$. Consequently, the section $\beta_P$ is different
from the pullback, to ${\mathcal R}_g$, of the above three sections on ${\mathcal M}_g$.

Another ${\mathcal C}^{\infty}$ section of projective structures on
${\mathcal R}_g$ was constructed in \cite{bgv} using the second order
theta functions. Also this section is not a pullback of a section of
projective structures on ${\mathcal M}_g$. In Theorem \ref{different}
it is shown that this section is also different from $\beta_P$.

The structure of the paper is as follows. Section \ref{se2} gives some preliminary notions on projective 
structures, on Schiffer variations and on the Prym map, and then the second fundamental form $\rho_P$ of the 
Prym map is described as a generalised Hodge Gaussian map, following \cite{cfadv} and \cite{cpt}. Section 
\ref{se3} is technical and it is entirely devoted to the construction of the section 
$\widehat{\eta}_P$. In Section \ref{se4} we prove that the section $\widehat{\eta}_P$ induces a projective 
structure on $C$ (Theorem \ref{betaP}). Section \ref{se5} contains Theorem \ref{prodotto}; also, the section 
$\widehat{\eta}_P$ is compared with the sections $\widetilde{\eta}$ and $\eta$ that govern the second 
fundamental forms of the Torelli maps for $\widetilde{C}$ and $C$ respectively (Theorem \ref{scomposizione}).
In Section \ref{se7} it is shown that $\widehat{\eta}_P$ induces a ${\mathcal C}^{\infty}$ section $\beta_P$ of the 
bundle of projective structures on ${\mathcal R}_g$, and also $\overline{\partial}(\beta_P)$ is computed.

\section{Preliminaries}\label{se2}

\subsection{Projective structure on compact Riemann surface}

Using the standard action of $\text{GL}(2, {\mathbb C})$ on ${\mathbb C}^2$,
the projective linear group $\text{PGL}(2, {\mathbb C})$ is identified with the group
of holomorphic automorphisms of ${\mathbb C}{\mathbb P}^1$.

Let $C$ be a compact Riemann surface of genus $g$, with $g \,\geq\, 2$. 

\begin{definition}
A projective structure on $C$ is given by a
holomorphic coordinate atlas $\{(U_i,\, \phi_i)\}_{i\in I}$ such that
$\bigcup_{i\in I} U_i\,=\, C$ and for every
pair $i,\,j\, \in\, I$ for which $U_i\bigcap U_j\, \not=\, \emptyset$, there is
a locally constant function $\tau_{j,i}\, :\, U_i\bigcap U_j\, \longrightarrow\,
\text{PGL}(2, {\mathbb C})$ satisfying the condition $\phi_j(z)\,=\,
\tau_{j,i}(z)(\phi_i(z))$ for all $z\, \in\, U_i\bigcap U_j$.
Two such collections of pairs $\{(U_i,\, \phi_i)\}_{i\in I}$ and
$\{(U_i,\, \phi_i)\}_{i\in J}$ are called equivalent if
their union $\{(U_i,\, \phi_i)\}_{i\in I\cup J}$ again satisfies the
the above condition. A \textit{projective structure} on $C$
is an equivalence class of collection of pairs satisfying the above condition (see \cite{Gu}).
\end{definition}

There are projective structures on every compact Riemann surface $C$, for example, locally 
defined sections of the uniformisation map $\widehat{C}\, \longrightarrow\, C$,
where $\widehat{C}$ is ${\mathbb H}$, $\mathbb C$ or ${\mathbb C}{\mathbb P}^1$,
produce a natural projective structure on $C$. 
In fact, the space of all projective structures on $C$ is an affine space modelled on the 
complex vector space $H^0(C,\, K^{\otimes 2}_C)$ \cite{Gu}.

Let us now explain another equivalent interpretation of projective structures on $C$. Consider the complex
surface $S \,=\, C\times C$. Denote by $\Delta$ the diagonal $\{(z,\, z)\, \,\big\vert\,\, z\, \in\,C\}
\,\subset\, S$, and define the holomorphic line bundle $M\,:=\, K_S(2\Delta)\,=\,
K_S\otimes_{{\mathcal O}_S}{\mathcal O}_S(2\Delta)$. Let
$p,\,q\,:\, S\, \longrightarrow\, C$ respectively be the first and second projection respectively.

For any $n\, \geq\, 1$, the $n$--th order infinitesimal neighborhood of $\Delta$ in
$C\times C$ will be denoted by $(n+1)\Delta$. Since $K_S \,\cong\,
(p^*K_C) \otimes (q^*K_C)$, one easily checks that $ M\vert_\Delta$ is the trivial line
bundle on $\Delta$. In fact, the trivialization
of ${ M}\vert_\Delta$ extends to a canonical trivialization of
${M}\vert_{2\Delta}$ (see \cite[p.~754, Theorem 2.1]{BR1}, \cite[p.~688, Theorem 2.2]{BR2}).
Let
\begin{equation}\label{evp}
s_0 \in\, H^0(2\Delta,\, {M}\vert_{2\Delta})
\end{equation}
be the section giving this
canonical trivialization.

Let
\begin{equation}\label{QC}
{\mathcal Q}(C)\, \subset\, H^0(3\Delta,\, {L}\vert_{3\Delta})
\end{equation}
be the locus of all sections $t$ whose restriction to $2\Delta$ satisfies the condition
$t\vert_{2\Delta}\,=\,s_0$ (see \eqref{evp} for $s_0$). The space ${\mathcal Q}(C)$ in \eqref{QC}
is an affine space modelled on the complex vector space $H^0(C,\, K^{\otimes 2}_C)$. 

Let ${\mathcal P}(C)$ denote the space of all projective structures on the
Riemann surface $C$. There exists a natural map
\begin{equation}\label{ephi}
{\mathcal P}(C)\, \longrightarrow\, {\mathcal Q}(C)
\end{equation}
which is an isomorphism of affine spaces modelled on 
$H^0(C,\, K^{\otimes 2}_C)$ (see \cite[p.~757, Theorem 3.2]{BR1} and \cite[p.~758, Lemma 3.6]{BR1};
see also \cite[p.~688, Theorem 2.2]{BR2}).

\subsection{Prym map}

Denote by $ {\mathcal R}_g$ the moduli space of all pairs of the form $(C,\, \pi)$, where $C$ is
a smooth projective curve of genus $g$ and $\pi\, :\, \widetilde{C} \,\longrightarrow\, C$ is a
nontrivial \'etale double cover. Such a pair $(C,\, \pi)$ gives a triple $(C,\, \alpha,\, \rho_\alpha)$
where $\alpha \,\in \,{\rm Pic}^0(C)$ satisfies the conditions $ \alpha^2 \,\cong\, {\mathcal O}_C\,\neq\, \alpha$
and $\rho_\alpha\, :\, \alpha^2\, \stackrel{\sim}{\longrightarrow}\, {\mathcal O}_C$ is a holomorphic
isomorphism. We will often suppress $\rho_\alpha$ for notational convenience.

Given a nontrivial \'etale
double cover $ \pi\,:\,\widetilde{C} \,\longrightarrow\, C$, consider the norm homomorphism
$$Nm_{\pi}\ :\ {\rm Pic}^0(\widetilde{C})\ \longrightarrow\ {\rm Pic}^0(C)$$ defined by
${\mathcal O}_{\widetilde{C}}(\sum_i a_i p_i)\,
\longmapsto\, {\mathcal O}_C(\sum_i a_i \pi(p_i))$, where $a_i\, \in\, {\mathbb Z}$ with $\sum_i a_i
\,=\, 0$ and $p_i\, \in\, \widetilde{C}$. The Prym variety $P(C,\alpha)$ is
the connected component, containing the origin, of the kernel of the homomorphism $Nm_{\pi}$. 
This $P(C,\alpha)$ is a principally polarized abelian variety of dimension $g-1$; the principal
polarization is denoted by $\Xi$. Denote by $\mathcal{A}_{g-1}$ the moduli space of principally polarized
complex abelian varieties of dimension $g-1$. The above construction defines the Prym map
\begin{equation}\label{prym}
{\mathcal P}\ :\ {\mathcal R}_g\ \longrightarrow\ {\mathcal A}_{g-1},\ \ [(C,\,\alpha)]\
\longmapsto\ [(P(C, \alpha), \Xi)].
 \end{equation}
 The Prym map $\mathcal P$ is generically an embedding for
 $g \,\geq\, 7$ \cite{friedman-smith}, \cite{kanev-global-Torelli},
 and it is generically finite for $g \,\geq\, 6$. Moreover, it is not
 injective in every genus and in fact it possesses positive dimensional
 fibers \cite{donagi}, \cite{mumford}, \cite{carlos}.

\subsection{Second fundamental form of the Prym map}

Assume that $g \,\geq\, 7$, and denote by ${\mathcal R}^0_g$ the Zariski open subset of $ {\mathcal R}_g$
where the Prym map is an embedding. On ${\mathcal R}^0_g$, there is the short
exact sequence of holomorphic vector bundle given by the differential of ${\mathcal P}$
\begin{equation}
\label{tangent}
0\ \longrightarrow\ T_{{\mathcal R}^0_g}\ \xrightarrow{\,\,\,d{\mathcal P}\,\,\,}
\ T_{{\mathcal A}_{g-1}}\big\vert_{{\mathcal R}^0_g}
\ \stackrel{p}\longrightarrow\ {\mathcal N}_{{\mathcal R}^0_g/{\mathcal A}_{g-1}}\ \longrightarrow\ 0,
\end{equation}
where $T_Y$ stands for the holomorphic tangent bundle of $Y$, and
${\mathcal N}_{{\mathcal R}^0_g/{\mathcal A}_{g-1}}$ is the normal bundle to
${\mathcal R}^0_g$.

Consider the Hermitian symmetric domain ${\mathcal H}_{g-1}\,\cong\, {\rm Sp}(2(g-1), {\mathbb
R})/{\rm U}(g-1)$, which is known as the Siegel space. The moduli space ${\mathcal A}_{g-1}$ is identified
with the quotient space ${\rm Sp}(2(g-1), {\mathbb Z})\backslash {\mathcal H}_{g-1}$.
Using this identification, the canonical
symmetric metric on ${\mathcal H}_{g-1}$ descends to an orbifold metric on ${\mathcal A}_{g-1}$; this
descended K\"ahler metric on ${\mathcal A}_{g-1}$ is called the Siegel metric. Denote by $\nabla$ the
Chern connection on $T_{{\mathcal A}_{g-1}}$ for the Siegel metric on ${\mathcal A}_{g-1}$. The second
fundamental form of ${\mathcal R}^0_g\, \hookrightarrow\, {\mathcal A}_{g-1}$ is the map:
$$
II_P\ :\ {\rm Sym}^2( T_{{\mathcal R}^0_g} )\ \longrightarrow\ {\mathcal N}_{{\mathcal R}^0_g/{\mathcal A}_{g-1}},\ \
X \otimes Y\ \longmapsto\ p(\nabla_X(Y)).
$$
Let
\begin{gather}\label{I2}
I_2(K_C\otimes \alpha)\ \,\subset\,\ {\rm Sym}^2(H^0(C,\ K_C \otimes \alpha))
\end{gather}
denote the kernel of the multiplication map
\begin{equation}\label{mm}
{\rm Sym}^2(H^0(C,\, K_C \otimes \alpha))\ \longrightarrow\
H^0(C,\, K_C^{\otimes 2})
\end{equation}
(recall that $\alpha^{\otimes 2}\,=\, {\mathcal O}_C$). The dual $(d{\mathcal P})^*$ of the differential of
the Prym map coincides with the multiplication map in \eqref{mm}. Consequently,
The conormal bundle of $\cP(\mathcal R^0_g) \,\subset\, \cA_{g-1}$ at the point $[(C,\, \alpha)]$
is identified with $I_2(K_C \otimes \alpha)$. Therefore, dualizing \eqref{tangent} at
a point $[(C,\, \alpha)] \,\in\, {\mathcal R}^0_g$ we get the short
exact sequence
\begin{equation*}
0\ \longrightarrow\ I_2(K_C \otimes \alpha)\ \longrightarrow\ {\rm Sym}^2 (H^0(C,\, K_C \otimes\alpha))
\ \xrightarrow{\,\,\,{(d{\mathcal P}})^*\,\,\,}\ H^0(C,\, K_C^{\otimes 2})\ \longrightarrow\ 0.
\end{equation*}
 
Thus, the dual of the second fundamental form of the Prym map $II_P$ at $[(C,\, \alpha)]$ is a linear map 
\begin{equation}\label{secondP}
 \rho_P\ :\ I_2(K_C \otimes \alpha)\ \longrightarrow\ {\rm Sym}^2 (H^0(C,\, K_C^{\otimes 2})).
 \end{equation}
Using Serre duality $H^0(C,\, K_C^{\otimes 2})\,=\, 
H^1(C,\, T_C)^*$,
we can also view $\rho_P$ as a linear map 
\begin{equation}\label{nrp}
\rho_P\ :\ I_2(K_C \otimes \alpha)\otimes H^1(C,\, T_C)\ \longrightarrow\ H^0(C,\, K_C^{\otimes 2}).
\end{equation}

In \cite{cfadv} it was proved that the map $\rho_P$ coincides with a generalisation of
the Hodge Gaussian maps given in \cite{cpt}. The map $\rho_P$ will
be described in Section \ref{se2.4}.

\subsection{Schiffer variations}
\label{section:Schiffer}

Consider an irreducible complex projective curve $C$ of genus at least one, and take a point
$x \,\in\, C$. Then the coboundary map of cohomologies for the short exact sequence
$$0\, \longrightarrow \,T_C \,\longrightarrow\, T_C(x) \,\longrightarrow\, T_C(x)_{|x}\,=\,
(T^{\otimes 2}_C)_x \,\longrightarrow\,0, $$
where $T_C$ is the holomorphic tangent bundle of $C$, yields an injection
\begin{equation}\label{b1}
H^0(T_C(x)_{|x})\ \hookrightarrow\ H^1(C,\,T_C).
\end{equation}
 
Elements of $H^1(C,\,T_C)$ lying in the image of the homomorphism in \eqref{b1}
are called \emph{Schiffer variations} at $x$. If
$(U,\, z)$ is a holomorphic coordinate chart centered at $x\, \in\, C$ and
\begin{equation}\label{bx}
b_x\ :\ C\ \longrightarrow\ {\mathbb R}
\end{equation}
is a bump 
function which is supported on $U$ and is equal to 1 on a neighbourhood of $x$, then
\begin{equation}\label{etheta}
\theta_x\,\ =\,\ \frac{\overline{\partial} b_x}{z} \otimes \frac{\partial}{\partial z}
\end{equation}
is a Dolbeault representative of a Schiffer variation at $x$. The map
\begin{equation}\label{Schiffer}
\xi_x\, :\, (T_C^{\otimes 2})_x\, \longrightarrow \,H^1(C,\, T_C),\ \ \ 
\lambda \left(\frac{\partial}{\partial z} \right)_x^{\otimes 2}\,\longmapsto\,\lambda [\theta_x]
\end{equation}
actually does not depend on the choice of the coordinate function $z$.

\subsection{The local system $\bC_\alpha$}\label{local-system}

Let $C$ be as above. Given a holomorphic line bundle $\alpha$ of
degree zero on $C$, there exists a Hermitian metric $H$ on $\alpha$
together with a flat connection $D_H$ on $\alpha$ which is compatible
with both the holomorphic structure and the Hermitian metric (see
e.g. \cite{gh}). The metric $H$ is unique up to multiplication by a
positive constant. This implies that the connection $D_H$ is unique. Take
a nontrivial holomorphic line bundle $\alpha$ of order two. Fix a holomorphic
isomorphism
$\rho_\alpha\, :\, \alpha^{\otimes 2}\,
\stackrel{\sim}{\longrightarrow} \, {\mathcal O}_C$, and let
$\pi\,:\, \widetilde{C}\, \longrightarrow \,C$ be the associated
unramified nontrivial double covering of $C$; note that
$\widetilde{C}\,=\, (\rho_\alpha)^{-1}(1_C)\, \subset\, \alpha$, where $1_C$ is the constant
function $1$ on $C$. In this case, we can take local
holomorphic trivializations $\{(U_{i},\, s_{i})\}$ of $\alpha$ such
that the images of the sections $s_{i}$ actually lie in
$\widetilde{C}\, \subset\, \alpha$. Then the cocycle
$g_{i,j} \,=\, s_{i}/s_{j}$ has values in $\{\pm1\}$, so it induces
a flat Hermitian metric, and a compatible flat connection, on $\alpha$.
This Hermitian metric given by the cocycle coincides with $H$ up to
the multiplication by a constant, and the flat connection coincides with $D_H$ (see
\cite{ko}).

If $\nu$ is a holomorphic frame of $\alpha$ over an open subset 
$U\, \subset\, C$, then the metric is encoded in the positive
function $h\,:=\,||\nu||_H^2$ and on $U$ we have
\begin{gather}\label{DpH}
D'_H (f \nu)\ =\ \left( \partial f + f \frac{\partial h}{h}\right) \otimes \nu
\end{gather}
for any smooth function $f$, where $D'_H$ is the $(1,\,0)$ type component
of $D_H$. The flatness condition of $D_H$ is equivalent to
the fact that the function $\partial h / h$ is holomorphic.

Write $D_H \,=\, D'_H\, +\, \overline{\partial}_\alpha$, where $\overline{\partial}_\alpha$ is the
Dolbeault operator on $\alpha$ (the $(0,\,1)$ type component of $D_H$). Such a
pair $(\alpha,\,H)$, which is also called a harmonic line bundle,
possesses the following properties (see \cite[Section 4.C and Remarque
8.1]{de}):

\begin{itemize}
\item The $D_H$--flat sections of $\alpha$ form a local system on $C$;
this local system is denoted by $\bC_\alpha$.

\item There is a de Rham complex for $\alpha$-valued forms:
\begin{gather*}
0\, \la\, \cA^0(\alpha) \,\xrightarrow{\,\,\,D_H\,\,\,}\, \cA^1(\alpha)\,\xrightarrow{\,\,\,D_H\,\,\,}
\, \cA^2(\alpha) \,\la\, 0.
\end{gather*}
\item The cohomology of this complex, denoted by
$H^p_{D_H}(C, \,\bC_\alpha) $, is isomorphic to the cohomology of the local system $\bC_\alpha$.
\item The analogue of the Hodge theorem holds; in other words,
there is an associated harmonic decomposition
$${\mathcal A}^{\bullet}(\alpha)\, =\, {\mathcal H}^{\bullet}(\alpha) \oplus im(D_H) \oplus im(D^*_H)
\,=\, {\mathcal H}^{\bullet}(\alpha) \oplus im(\overline{\partial}_\alpha)
\oplus im(\overline{\partial}^*_\alpha),$$ where ${\mathcal H}(\alpha)$ is
the kernel of the Laplacian operator
$\Delta \,=\, D_H D^*_H + D^*_H D_H \,=\,
2(\overline{\partial}_\alpha\overline{\partial}^*_\alpha+
\overline{\partial}^*_\alpha\overline{\partial}_\alpha)$.

\item The K\"ahler identities hold for the operators $D'_H$ and $\debar$.

\item The following Hodge decomposition holds:
\begin{equation}
\label{hodgedec}
H^1_{D_H}(C,\, {\mathbb C}_{\alpha})\, \cong\,
H^{1,0}_{\overline{\partial}_\alpha}(C,\, {\alpha}) \oplus H^{0,1}_{\overline{\partial}_\alpha}
(C,\, {\alpha})\, \cong\, H^0(C,\,K_C \otimes \alpha) \oplus H^1(C,\, \alpha).
\end{equation}

\item The following principle of two types holds:
$$ker(D'_H) \cap ker(\overline{\partial}_\alpha) \cap (im(D'_H) + im(\overline{\partial}_\alpha))
\ =\ im(D'_H \overline{\partial}_\alpha).$$

\item The Mayer-Vietoris sequence holds for the cohomology $H^*_{D_H}$.
\item Poincar\'e duality holds:
$H^p_{D_H} (C,\, \bC_\alpha)\ \cong\ H^{2-p}_{D_H} (C,\, \bC_\alpha)^*$.
\end{itemize}

\begin{lemma}\label{hyperplane}
For any $x\,\in\, C$, the map
$ H^1_{D_H}(C,\, {\mathbb C}_{\alpha})\, \la \, H^1_{D_H}(C\setminus \{x\}
,\, {\mathbb C}_{\alpha})$ induced by the inclusion map $C\setminus \{x\}\, 
\hookrightarrow\, C$ is injective and it image has codimension 1.
\end{lemma}

\begin{proof}
Let $U$ be a disc around the point $x$, and let $V\,:=\, C\setminus\{x\}$.
Since $\bC_\alpha|_{U}$ is trivial, we have
$$H^0_{D_H}(U,\,\bC_\alpha)\, = \,H_{D_H}^0(U\cap V,\,\bC_\alpha)\, =\,
\bC_{\alpha,x}.$$ Poincar\'e duality gives that
$H_{D_H}^2(C,\,\bC_\alpha) \,\cong\, H_{D_H}^0(C,\, \bC_\alpha)^* $. Since $\alpha$ is non-trivial, the
latter space is trivial. Therefore, by applying
Mayer-Vietoris to $U,\,V$ we get the following:
\begin{gather*}
0 \,\la\, H^1_{D_H}(C,\, \bC_\alpha) \,\la\, H^1_{D_H}(U,\, \bC_\alpha)\oplus H^1_{D_H}(V,\,
\bC_\alpha)\,\la\, H^1_{D_H}(U\cap V,\, \bC_\alpha) \,\la\, 0.
\end{gather*}
Now $ H^1_{D_H}(U,\, \bC_\alpha) \,=\, 0$ because $U$ is simply connected and
$H^1_{D_H}(U\cap V,\, \bC_\alpha)\,=\, \bC_{\alpha,x}$.
Thus $H^1_{D_H}(C,\, \bC_\alpha)$ is a hyperplane in $H^1_{D_H}(U,\, \bC_\alpha)$, and the corresponding
quotient is the line $\bC_{\alpha,x}$.
\end{proof}

\subsection{Second fundamental form as Hodge Gaussian map}\label{se2.4}

In this section Hodge-Gaussian maps are recalled from \cite{cpt}. The 
second fundamental form (constructed in \eqref{secondP}) coincides with a
generalized Hodge-Gaussian map \cite{cfadv}, which is then used in
describing $\rho_P$.

We start by describing the procedure to construct Hodge-Gaussian maps.
Fix a basis $\{\omega_1,\,\cdots,\, \omega_{g-1}\}$ of
$H^0(C,\, K_C \otimes \alpha)$; write, locally,
\begin{gather}\label{fi}
\omega_i\ \,=\,\ f_i(z) dz \otimes \nu,
\end{gather}
where $\nu$ as before is a locally defined holomorphic frame of $\alpha$. Fix
\begin{gather}\label{representationQ}
Q\, = \,\sum_{i,j} a_{ij} \omega _i \otimes \omega _j \,\in\, I_2(K_C\otimes\alpha)
, \ \ \ \text{ and } \ \ \ \xi \,\in\, H^1(C,\, T_C),
\end{gather}
where $I_2(K_C\otimes\alpha)$ is defined in \eqref{I2}. Let $\widetilde \xi \,\in\, \cA^{0,1}
(T_C)$ be such that $[\widetilde \xi ]\, \,=\xi$. For all $1\,\leq\, i\, \leq\, g-1$, using
the natural duality pairing $T_C\otimes K_C\, \longrightarrow\, {\mathcal O}_C$, write
\begin{gather}\label{HG1}
\widetilde \xi \cdot \omega_i\ \,=\ \, \gamma_i + \overline{\partial}_\alpha h_i,
\end{gather}
where $\gamma_i \,\in\, {\mathcal A}^{0,1}(\alpha)$ is harmonic and
$h_i$ is a section of $\alpha$. Set
\begin{gather}\label{HG}
\HG (Q \otimes \xi )\ :=\ \sum_{i,j} a_{ij} \omega_i D'_H ( h_j).
\end{gather}
Since $\nu^{\otimes 2}$ is a scalar function, the right hand side is a
section of $K_C^{\otimes 2}$.

\begin{theorem}[\protect{\cite[Theorem 4.4]{cpt}}]
The map
\begin{gather*}
\HG\ :\ I_2(K_C\otimes \alpha) \otimes H^1(C,T_C)\ \la\ H^0(C,\,K_C^{\otimes 2})
\end{gather*}
is well-defined (see \eqref{I2} for $I_2(K_C\otimes\alpha)$);
in other words, it does not depend on the choice of the
representative $\widetilde \xi$, on the choice of $h_j$ and on the expression
\eqref{representationQ} of $Q$. Furthermore, $\HG$ is a linear map.
\end{theorem}

\begin{theorem}[\protect{\cite{cfadv}}]
With $\rho_P$ being understood as in \eqref{nrp},
$$\rho_P\ =\ -2 \HG.$$
\end{theorem}

See the proof of Theorem 2.1 in \cite{cfadv}.

We apply this to the case of a Schiffer variation
as in \eqref{Schiffer}. Let $b_x$ and $ \theta_x$ be as in
\eqref{bx} and \eqref{etheta} respectively. Then
\begin{equation}\label{f2}
\gamma_i + \overline{\partial}_\alpha h_i
\,\, =\,\, \theta_x ( \omega_i)\nu\,\, =\,\, 
\frac{\overline{\partial} b_x}{z-z(x)} f_i(z) \nu\,\, =\, 
\, \overline{\partial}(\frac{b_x f_i}{z-z(x)}) \otimes \nu.
\end{equation}
Set
\begin{gather} \label{def:g_i} g_i\,: =\, \frac{b_x f_i}{z-z(x)} \cdot
\nu - h_i, \qquad \eta_i\,:=\, D'_H g_i.
\end{gather}
On $C \setminus \{x\}$ we have
$\gamma_i \,=\, \overline{\partial}_\alpha g_i$, and
\begin{gather}\label{exact}
\gamma_i + \eta_i \ =\ \overline{\partial}_\alpha g_i + D'_H g_i \ =\ D_H g_i.
\end{gather}
It follows from \eqref{def:g_i} that
\begin{gather*}
D_H \, \eta_i \,=\,D_H ( D'_H g_i)\,=\, {\debar}_\alpha D_H' g_i\, =\, - D_H' {\debar}_\alpha g_i
\, =\, -D'_H \gamma _ i\, =\, 0
\end{gather*}
because $\gamma_i$ is harmonic. So $\eta_i$ defines a class in
$H^1_{D_H}(C\setminus \{x\},\, \bC_\alpha)$. From \eqref{exact} it follows
that $[\eta_i]\, =\, - [\gamma_i]$ as elements of
$H^1_{D_H}(C\setminus \{x\},\, \bC_\alpha)$. Moreover, $[\gamma_i]$
lies in the subspace $H^{0,1}_\debar (C,\,\alpha) \,\subset\, H^1_{D_H}(C,\, \bC_\alpha)$,
which is a hyperplane by Lemma \ref{hyperplane}, so we have
\begin{equation}\label{etai}
[\eta_i]\, =\, - [\gamma_i] \,\in\,
H^{0,1}_{\overline{\partial}}(C,\,\alpha)\, \subset\,
H^1_{D_H}(C,\, {\mathbb C}_{\alpha}) \,\subset
\, H^1_{D_H}(C\setminus \{x\},\, {\mathbb C}_{\alpha}).
\end{equation}

\begin{theorem}\label{ths}
Given a quadric
$Q\, =\, \sum_{i,j =1}^{g-1} a_{ij} (\omega_i \otimes \omega_j )
\,\in\, I_2(K_C \otimes \alpha)$ and $x\,\in\, C$,
fix a holomorphic coordinate function $z$ around $x$ and write locally
the sections $h_i$ in \eqref{f2} as $h_i\, =\, l_i \nu$, where $l_i$ is a function. Then
\begin{gather}\label{rho-eta}
\rho_P(Q\otimes\xi_x(\frac{\partial}{\partial z}(0)^{\otimes 2})) \ =\ -\sum_{i,j} a_{ij} \omega_i\eta_j,\\
\label{eta_i}
\eta_i\ =\ \left(- \frac{f_i(z(x))}{(z-z(x))^2} + \frac{\partial
(\log l_i)}{\partial z}(z(x))\frac{f_i(z(x))}{z-z(x)} +
\psi_i(z)\right) \nu \otimes dz,
\end{gather}
where $\rho_P$, $f_i$, $\xi_x$ are as in \eqref{nrp}, \eqref{fi},
\eqref{Schiffer} respectively, and $\psi_i$ is some local holomorphic function.
\end{theorem}
Formula \eqref{rho-eta} follows from the same argument in Lemma 3.2 of
\cite{cpt} using $\alpha$ and $D_H$ instead of the trivial line bundle and
the de Rham differential $d$. Formula \eqref {eta_i} is proved in \cite[p.~145]{cpt}

\begin{proposition}
Assume that $x\,\in\, C$ is not a base point of $K_C\otimes \alpha$.
Then there is a well-defined element
\begin{equation}\label{sP2}
\s{z,\nu,x}\ \, \in\ \,H^0(C,\, K_C \otimes \alpha(2x))
\end{equation}
such that for any $i$,
\begin{gather}\label{eta_isP}
\eta_i\ =\ f_i(z(x)) \s{z,\nu,x} .
\end{gather}
Furthermore, near $x$,
\begin{equation}\label{eta_p}
s_{z,\nu,x}\ =\
\left(- \frac{1}{(z- z(x))^2} +
\frac{1}{z- z(x)}\frac{\partial (\log l_i)}{\partial z}(z(x)) + \phi(z)\right) \nu \otimes dz,
\end{equation}
where $\phi(z)$ is holomorphic function and $z$ is a holomorphic coordinate function around $x$.
\end{proposition}

\begin{proof}
It follows from \eqref{eta_i} that the combination
$f_j(z(x)) \eta_i - f_i(z(x)) \eta_j$ has a continuous extension to
$C$ and this extension is $D_H$-harmonic. By \eqref{exact},
\begin{equation*}
f_j(z(x)) \eta_i - f_i(z(x)) \eta_j + f_j(z(x)) \gamma_i - f_i(z(x)) \gamma_j
\ =\ D_H \left(f_j(z(x)) g_i - f_i(z(x)) g_j \right).
\end{equation*}
Since the left-hand side is $D_H$-harmonic and the right-hand side
is $D_H$-exact, both actually vanish. In particular, the $(0,1)$-part of the
left-hand side is zero; thus
\begin{gather}\label{eq:4}
f_j(z(x)) \eta_i \ = \ f_i(z(x)) \eta_j.
\end{gather}
By assumption, there is $i$ such that $f_i(z(x)) \,\neq\, 0$. Set
\begin{gather}\label{eta:1}
\s{z,\nu,x}\ :=\ \frac{\eta_i}{f_i(z(x))} , \qquad \phi\ :=\ \psi_i / f_i(x).
\end{gather}
It follows from \eqref{eq:4} that these definitions are independent
on the choice of $i$. Moreover, if $f_j(z(x))\,=\,0$, then $\eta_j\,=\,0$.
This proves \eqref{eta_isP}. Formula \eqref{eta_p} follows
immediately from the definition.
\end{proof}

By the definition of $\s{z,\nu,x}$ and \eqref{rho-eta} we have 
\begin{equation}\label{rho_P}
\rho_P(Q\otimes\xi_x(\frac{\partial}{\partial z}(0)^{\otimes 2}))
\ =\ - \s{z,\nu,x} \cdot \sum_{i,j} a_{ij} f_i(z(x))
\omega_j\ \in\ H^0(C,\, K_C^{\otimes 2}).
\end{equation}

The expression in \eqref{eta_p} clearly depends on the the choices of
$z$ and $\nu$, but in a rather simple way.

\begin{proposition}
Fix a holomorphic coordinate function $z$ and a holomorphic frame $\nu$ of $\alpha$ around $x\, \in\, C$.
Let $\nu^*$ denote the dual frame. The map
\begin{gather}
\eta_{P,x}\,:\, T_xX \otimes \alpha_x\, \la \, H^0(C,\, K_C\otimes \alpha (2x)), \label{a2}\\
\notag \lambda \frac{\partial}{\partial z } \otimes
\nu \ \longmapsto\ \lambda \s{z,\nu,x},
\end{gather}
is independent of the choices of $z$ and $\nu$. Moreover, it is linear, hence
\begin{gather*}
\eta_{P,x} \ \in\ H^0(C,\, K_C \otimes \alpha(2x))\otimes
(K_C\otimes\alpha^*)_x.
\end{gather*}
\end{proposition}

\begin{proof}
The argument is simple, but the computation is somewhat cumbersome.

Let $z'$ (respectively, $\nu'$) be another choice of holomorphic coordinate
function (respectively, holomorphic frame).
Without loss of generality it may be assumed that $z(x) \,=\, z'(x)\, =\, 0$; this assumption is imposed.
Let
\begin{gather}\label{zpu}
z'\ =\ \varphi(z) , \qquad \nu'\ =\ e^u \nu
\end{gather}
for holomorphic functions $\varphi$ and $u$ with $\varphi(0) \,=\, 0$
and $ \varphi' \,\neq\, 0$. In order to compare $\s{z,\nu,x}$ and
$\s{z', \nu', x}$, it is needed to trace through their entire construction
starting from \eqref{f2}. Choose a bump
function $b_x$ as before and set
\begin{gather*}
\widetilde\xi\ :=\ \frac{\overline{\partial} b_x}{z} \otimes
\frac{\partial}{\partial z} , \qquad \widetilde\xi'\ :=\
\frac{\overline{\partial} b_x}{z'} \otimes
\frac{\partial}{\partial z'}.
\end{gather*}
We start by comparing $\widetilde\xi$ and $\widetilde{\xi}'$. Since
$\varphi ( 0)\, =\, 0 $, setting $a\, :=\, \varphi'(0) \,\neq\, 0$, we have
\begin{gather*}
z'\ =\ \varphi (z)\ =\ z (a + \psi) 
\end{gather*}
for a holomorphic function $ \psi$ with $\psi(0) \,=\, 0$.
Furthermore,
\begin{gather*}
\frac{\partial}{\partial z' }\ =\ \frac{1}{\varphi'(z)}
\frac{\partial}{\partial z } ,\quad \text{ so } \quad \widetilde\xi'\ =\
\frac{{\overline{\partial} b_x}}{z(a+\psi) \varphi' } \otimes \frac{\partial}{\partial z}.
\end{gather*}
As $a\,=\, \varphi'(0)$ and $\psi(0)\,=\, 0$, we have
\begin{gather}\label{aw}
\frac{1}{(a+\psi) \varphi' }\ =\ \frac{1}{a^2 } + zw
\end{gather}
for a holomorphic function $w$, and hence
$$
\widetilde{\xi} '\,\ =\,\ \frac{1}{a^2} \widetilde \xi + \debar (b_x w) \frac{\partial}{\partial z}.
$$
We proceed with the description of Hodge-Gaussian maps:
\begin{gather*}
\widetilde \xi ' \cdot \omega_i\ =\ \gamma_i' + \debar h'_i , \qquad
\widetilde \xi \cdot \omega_i\ =\ \gamma_i + \debar h_i .
\end{gather*}
So we can choose $h_i'\, =\, \frac{h_i}{a^2} + b_xwf_i$, and hence
$g_i '\, =\, \frac{b_x f'_i }{z'} \nu' - h_i '$.
Now \eqref{fi} and \eqref{zpu} imply that
\begin{equation}\label{fip}
f_i'\, =\, \omega_i \biggl ( \frac{\partial }{\partial z'} \otimes
\nu'^* \biggr ) \,=\, \frac{1}{\varphi' e^u } f_i, \ \ \,
g_i ' \,=\, \frac{b_x f_i} { z (a + \psi) \phi'} - \frac{1}{a^2} h_i - b_xwf_i. \nu .
\end{equation}
Eq. \eqref{aw} implies that $g_i' \,=\, \frac{1}{a^2} g_i$, $\eta_i' \,=\, \frac{1}{a^2} \eta_i$.
Eq. \eqref{fip} implies that
\begin{equation}\label{ssp}
\s{z',\nu',x}\,\ =\,\ \frac{e^{u(0)}}{a} \s{z,\nu,x}.
\end{equation}
Assume now that
\begin{gather*}
\lambda \frac{\partial}{\partial z}(x) \otimes \nu(x)\ =\ \lambda'
\frac{\partial}{\partial z'}(x) \otimes \nu'(x) .
\end{gather*}
Then using \eqref{ssp} it follows immediately that
\begin{gather*}
\lambda\ =\ \lambda ' \frac{e^{u(0)}} {a} , \quad \text { so }
\lambda \s{z,\nu,x}\ =\ \lambda ' \s{z',\nu',x}.
\end{gather*}
Thus, $\eta_{P,x}$ is well-defined.
\end{proof}

Using the Poincar\'e adjunction formula
${\mathcal O}_C(x)_{|x}\,=\, (T_C)_x$ it follows that
$(K_C \otimes \alpha(2x))_x\otimes (K_C\otimes\alpha^*)_x\,=\,\mathbb
C$. We note that $\eta_{P,x}$ takes the value $1$ at the point $x$.

\section{Construction of a twisted meromorphic bidifferential}\label{se3}

Take $x\, \in\, C$ and
consider a section $\beta \,\in\, H^0(C, \,K_C \otimes \alpha(2x) )$. Then
$\beta \in \cA^{1,0}(C\setminus \{x\}, \alpha)$ and
$\overline{\partial} \beta\, =\,0$ on $C \setminus \{x\}$. Moreover
$D'_H \beta \,=\,0$ because it is a $(2,\,0)$-form on a curve. Thus
$D_H \beta \,=\, 0$, and hence the following homomorphism is obtained:
\begin{equation}\label{j_x}
(j_P)_x\,:\, H^0(C, \,K_C \otimes \alpha(2x) )\, \longrightarrow\, H^1_{D_H}(C \setminus \{x\}, \,
{\mathbb C}_{\alpha}),\ \ \beta\,\longmapsto\, [\beta].
\end{equation} 
\begin{lemma}
The map $(j_P)_x$
is injective unless $C$ is
hyperelliptic and $x$ is a Weierstrass point such that
$\alpha\, =\, {\mathcal O}_C(x-y)$ for some other Weierstrass point
$y\, \in\, C$. In particular, $(j_P)_x$ is injective for all $x$ if
$C$ is not hyperelliptic.
\end{lemma}

\begin{proof}
Take a non-zero $\beta \,\in\, H^0(C,\, K_C \otimes \alpha(2x) )$
and assume that $(j_P)_x (\beta)\, =\, [\beta] \,= \,0$. This implies that there
is a ${\mathcal C}^{\infty}$ section $f$ of $\alpha$ over
$C\setminus\{x\}$ for which $\beta \,=\, D_H (f)$. On
$C\setminus \{x\}$, the section $\beta$ is a holomorphic
$(1,\,0)$-form with values in $\alpha$, so
$\overline{\partial}_\alpha f \,=\,0$ and $\beta\, =\, D'_H(f)$.

In principle $f$ could have an essential singularity at $x$. It will be proved
that $f$ has at most a simple pole.

Fix a holomorphic coordinate function $z$ on a simply connected open neighbourhood
$U\, \subset\, C$ of $x$ with $z(x)\,=\,0$. Fix a holomorphic frame $\nu$ of
the restriction of $\alpha$ to $U$. Write $f\, =\, l\cdot \nu$ and
$\beta \,=\, (z^{-2} g) dz \otimes \nu$ with $g$ holomorphic on $U$ and $l$ holomorphic
to $U\setminus\{x\}$. Using \eqref{DpH} the condition $\beta \,=\, D'_H(f)$ translates to
\begin{gather*}
l' + \lambda l\ =\ \frac{g}{z^2}, \qquad \text{ where }\ \lambda\ :=\ \frac{1}{h} \, \frac{\partial
h}{ \partial z}.
\end{gather*}
Recall that $\lambda$ is holomorphic on $U$ (see \eqref{local-system}).
Take $u\,\in\, \OO(U)$ such that $u' \,=\, \lambda$. Set
$ v(z) \,: =\, l(z) e^{u(z)}$. Then
\begin{gather*}
v'\ =\ e^u( l' + \lambda l)\ =\ z^{-2} e^u g.
\end{gather*}
So $v'$ has at most a double pole at $z\,=\,0$ and $v$ has at most a
simple pole at $0$. The same holds for $l\, =\, v e^{-u}$, so $f$ is
meromorphic with at most a simple pole at $x$; in other words,
$f\, \in\, H^0(C, \,\alpha(x))$. Since $\beta \,\neq\, 0$, we also have
$f \,\neq\, 0$. Hence there exists a point $y \,\in\, C\setminus \{x\}$ such that
$\alpha \,\cong\, {\mathcal O}_C( y-x)$. This implies that
${\mathcal O}_C(2y-2x) \,=\, {\mathcal O}_C$. Consequently, $C$ is hyperelliptic with $x$ and
$y$ being Weierstrass points of $C$.
\end{proof}

Consider the complex surface $S\,:=\, C\times C$, and denote by $\Delta$ the reduced
diagonal divisor $\{(y,\, y)\,\, \big\vert\,\, y\, \in \, C\}$ on
$S$. Let $p,\, q\,:\, S \,\longrightarrow\, C$ be the projections with
$p(x,\,y)\, =\,x$ and $q(x,\,y) \,=\, y$.

As before, fix a nontrivial holomorphic line bundle $\alpha$ on $C$ together
with a holomorphic isomorphism of $\alpha^{\otimes 2}$ with ${\mathcal O}_C$; this
isomorphism produces a holomorphic isomorphism between $\alpha$ and $\alpha^*$.

We have $ K_S\,\cong\,(p^*K_C) \otimes (q^*K_C)$.
Consider the line bundles on $S$
\begin{gather}
\notag \alpha_S\,:=\, (p^*\alpha) \otimes (q^*\alpha),\\
\label{del}
L\ :=\ K_S\otimes \alpha_S(2\Delta)\ =\ K_S\otimes \alpha_S
\otimes {\mathcal O}_S(2\Delta).
\end{gather}
Set
\begin{equation}\label{VP}
V_P\ :=\ p_*(q^*(K_C \otimes \alpha) ( 2\Delta)),\ \ \, E\ :=\ p_* L.
\end{equation}
The projection formula gives that $E \,=\, K_C \otimes \alpha \otimes V_P$. For $x\, \in\, C$, we have
\begin{gather*}
(q^*(K_C \otimes \alpha))(2 \Delta)\big\vert_{\{x\}\times C }\,=\, K_C\otimes \alpha (2x)\,:=\,
K_C\otimes \alpha \otimes {\mathcal O}_C(2x), \\
H^0(p^{-1}(x),\, q^*(K_C \otimes \alpha) (2\Delta)\big\vert_{p^{-1}(x)}) \,\cong\,
H^0 (C,\,K_C \otimes \alpha (2x)).
\end{gather*}
By Grauert's
semicontinuity theorem, $V_P$ is a holomorphic vector bundle over $C$
with fiber
$$(V_P)_x \,\cong\, H^0(C ,\, K_C\otimes \alpha (2x))$$ over $x\, \in\, C$.
For every $x \,\in\, C$, the map $\eta_{P,x}$ defined in \eqref{a2}
is an element of
$K_{C,x}\otimes \alpha _x \otimes H^0(C ,\, K_C\otimes \alpha
(2x))$. So this defines a ${\mathcal C}^{\infty}$ section $\eta_P$ of
the vector bundle
$E \,=\, (K_C \otimes \alpha)\otimes p_* (q^*(K_C \otimes
\alpha)(2\Delta))$. Denote by
\begin{equation}\label{a3}
(i_P)_C\ :\ H^1(C,\, \alpha)
\ \cong\ H^{0,1}_{\overline{\partial}_\alpha}(C,\,\alpha) \ \hookrightarrow\ H^1(C,\, {\mathbb C}_{\alpha})
\end{equation}
the inclusion map given by the Hodge decomposition, and by
\begin{equation}\label{a4}
(\psi_P)_x\ :\ T_x C\otimes \alpha_x\ \longrightarrow\ H^1(C,\,\alpha) \cong H^0(C,\, K_C \otimes \alpha)^*
\end{equation}
the dual of 
the evaluation map $ev_x\,:\, H^0(C,\, K_C \otimes \alpha)\, \longrightarrow\,
(K_C \otimes \alpha)_x$.
We have the following inclusions
\begin{equation*}
H^0(C,\, K_C \otimes \alpha) \,\subset\, H^0(C,\, (K_C \otimes \alpha)(x))\,\subset\, H^0((C,\,
K_C \otimes \alpha)(2x) )
\, \stackrel{(j_P)_x}{\hookrightarrow}\, H^1_{D_H}(C\setminus\{x\},
\, {\mathbb C}_{\alpha}) .
\end{equation*}
By Riemann-Roch $h^0(K_C \otimes \alpha) \,=\, g-1$,
$h^0((K_C \otimes \alpha)(x) )\,=\, g$ and
$h^0((K_C \otimes \alpha)(2x)) \,= \,g+1$. It follows from
\eqref{etai}, and \eqref{eta_isP} that
$(j_P)_x(\suz) \,=\, [\suz] \, \in\, H^{0,1}_\debar(C,\, \alpha)\, \subset\, H^1(C,\, {\mathbb
C}_{\alpha})$ (the inclusion follows from \eqref{hodgedec}). Thus,
we have
\begin{equation}\label{classe} 
[\suz]\ \in\ (j_P)_x^{-1}(H^1(C,\, \alpha)).
\end{equation}
Set
\begin{gather}
\notag
{\mathcal H}_x:\ =\ H^0(C,\, K_C \otimes \alpha) \oplus \langle
\suz \rangle \ \subset\ H^0(C,\, K_C \otimes \alpha)(2x)),\\
\label{a1}
(\chi_P)_x\ :=\ {(j_P)_x}\big\vert_{{\mathcal H}_x}\ :\ {\mathcal H}_x\ \longrightarrow\
H^1(C,\, {\mathbb C}_{\alpha}).
\end{gather}

\begin{proposition}\label{diagx}
The following diagram is commutative:
\begin{equation}\label{cd}
\begin{tikzcd}[column sep=large] 
{} & H^1(C,\,\alpha) \arrow[hook]{rd}{(i_P)_C} & \\
T_x C\otimes \alpha_x \arrow{ru}{(\psi_P)_x} \arrow{rd}[swap]{-({\eta_P})_x} & 
& H^1(C,\, {\mathbb C}_{\alpha}) \\
& {\mathcal H}_x \arrow{ru}{(\chi_P)_x} &
\end{tikzcd}
\end{equation}
(the above maps are defined in \eqref{a1}, \eqref{a2}, \eqref{a3} and \eqref{a4}). 
\end{proposition}

\begin{proof}
Consider the short exact sequence
$$ 0\, \longrightarrow\, \alpha\, \longrightarrow\, \alpha(x)\,=\,
\alpha\otimes{\mathcal O}_C(x) \,\longrightarrow\, \alpha(x)\big\vert_{x}\,=\, \alpha_x\otimes
T_xC\, \longrightarrow\, 0$$
and the corresponding long cohomology exact sequence
\begin{equation}\label{edelta}
0\,\longrightarrow\, H^0(\alpha(x)\big\vert_{x})\,=\, \alpha_x\otimes T_xC
\, \stackrel{\delta}{\longrightarrow}\, H^1(C,\, \alpha).
\end{equation}
Here $H^0(C,\, \alpha(x))\,=\,0$ because ${\rm genus}(C)\,\geq\, 1$; the above isomorphism
$\alpha(x)\big\vert_{x}\,=\, \alpha_x\otimes
T_xC$ is given by the Poincar\'e adjunction formula.
The image of $\delta$ in \eqref{edelta} is generated by
the class of
\begin{gather*}
\gamma_x\ :=\ \overline{\partial}(\frac{b_x }{z- z(x)}) \otimes \nu\ \in\ \mathcal A^{0,1}(\alpha),
\end{gather*}
 where $b_x$ is a compactly
supported function defined on a small open neighborhood $U$ of $x$ with $b_x \,\equiv\, 1$ on a smaller
neighborhood of $x$, $z$ is a holomorphic
coordinate function on $U$ and $\nu$ is a holomorphic frame of $\alpha$ over $U$.

Fix a basis $\{\omega_1,\,\cdots ,\, \omega_{g-1}\}$ of $H^0(C,\, K_C \otimes \alpha)$.
Choose $i$ such that $\omega_i(x) \,\neq\, 0$. Then by definition \eqref{eta:1}
\begin{gather*}
\suz\ =\ \frac{\eta_i }{ f_i(z(x))}.
\end{gather*}
By \eqref{etai} and \eqref{f2}
$$
[\eta_i] \,=\,-[\gamma_i] \,=\, -[\gamma_i + \debar_\alpha h_i ]
 \, =\, - \biggl [ \debar_\alpha \biggl ( \frac{b_x f_i}{z-z(x)} \otimes \nu \biggr ) \biggr].
$$
This implies that
$$
(j_P)_x ( \suz) \,=\, [\suz] \,=\, - \frac{1}{f_i(z(x))}
\biggl [ \debar_\alpha \biggl ( \frac{b_x f_i}{z-z(x)} \otimes \nu \biggr )
\biggr ]
$$
$$
 =\, - \biggl [ \debar \biggl ( \frac{b_x}{z-z(x)}\biggr )\otimes \nu 
\biggr ]\, =\, - [\gamma_x].
$$
Therefore, if $\omega_i \,=\, f_i(z) dz \otimes \nu$ around $x$, we have 
\begin{equation*}
\int_C (j_P)_x(\suz) \wedge \omega_j\ =\ \int_C \overline{\partial} \biggl (\frac{b_x }{z- z(x)} \biggr)
\otimes\nu\wedge \omega_j\ =\ f_j(z(x)).
\end{equation*}
It follows that
$$-(j_P)_x \circ ((\eta_P)_x (\frac{\partial}{\partial z}(x) \otimes \nu_x))\ =\
(i_P)_C \circ ((\psi_P)_x(\frac{\partial}{\partial z}(x) \otimes \nu_x))\ \in\ H^1(C,\, \alpha).$$
This completes the proof. 
\end{proof}

The diagram \eqref{diagx} is a Prym version of a similar diagram for
Jacobians of curves \cite{bcfp}. In other words, there are two similar diagrams,
one for $\widetilde{C}$ and one for $C$. We will now recall them and
relate then with \eqref{diagx}.

Consider the \'etale double cover
\begin{equation}\label{dc}
\pi\ :\ \widetilde{C}\ \longrightarrow\ C
\end{equation}
given by the nontrivial line bundle $\alpha$ and the trivialization of
$\alpha^{\otimes 2}$. Denote the complex surface
$\widetilde{S}\,=\, \widetilde{C}\times \widetilde{C}$; the diagonal
divisor $\{(z,\, z)\,\,\big\vert\,\,\, z\,\in\, \widetilde{C}\}$ in
$\widetilde{C}\times\widetilde{C}$ is denoted by
$\widetilde{\Delta}$. For any $y\,\in\, \widetilde{C}$ the following
commutative diagram was established in \cite[Proposition 3.5]{bcfp}:
\begin{equation}\label{cdtilde}
\begin{tikzcd}[column sep=large] 
{} & H^1(\widetilde{C},\,{\mathcal O}_{\widetilde{C}}) \arrow[hook]{rd}{i_{\widetilde{C}} }& \\
(T_{\widetilde C})_y\arrow{ru}{\widetilde{\psi}_y} \arrow{rd}[swap]{-\widetilde{\eta}_y}
& & H^1(\widetilde{C},\, {\mathbb C}) \\
& H^0({\widetilde C},\, K_{\widetilde C}(2y)) \arrow[hook]{ru}{\widetilde{j}_y} &
\end{tikzcd}
\end{equation}
where the map $i_{\widetilde{C}}$ is given by the Hodge decomposition,
$\widetilde{\psi}_y$ is the dual of the evaluation map on 1-forms at
$y$ and ${\widetilde{j}_y}$ 
associates to a meromorphic form its de Rham cohomology class; in \eqref{cdtilde},
$\widetilde{\eta} \,\in\, H^0(\widetilde{S},\, K_{\widetilde S}(2
\widetilde{\Delta}_{\widetilde{S}}))$ is a meromorphic form
that gives the second fundamental form of the Torelli map for the curve
$\widetilde C$ (see \cite[Theorem 3.7]{cfg}).

Similarly, we have the analogous diagram --- for the curve $C$ and a
form $\eta \,\in\, H^0(S,\, K_{S}(2 {\Delta}_{S}))$ --- that gives the
second fundamental form of the Torelli map for $C$:
\begin{equation}\label{cdC}
\begin{tikzcd}[column sep=large] 
{} & H^1(C,\,{\mathcal O}_{C}) \arrow[hook]{rd}{i_{C} }& \\
(T_C)_x\arrow{ru}{\psi_x} \arrow{rd}[swap]{-\eta_x} & & H^1(C,\, {\mathbb C}) \\
& H^0(C,\, K_{C}(2x)) \arrow[hook]{ru}{j_x} &
\end{tikzcd}
\end{equation}

To investigate the relationship between the three diagrams in \eqref{cd}, \eqref{cdtilde} and \eqref{cdC},
first note that
\begin{equation}\label{dis}
\pi_*{\mathcal O}_{\widetilde{C}}\ =\ {\mathcal O}_C\oplus\alpha ,
\end{equation}
where $\pi$ is the map in\eqref{dc}. Since $\pi$ is \'etale, the projection formula gives
$$
\pi_* T_{\widetilde{C}}\,=\, \pi_* \pi^* T_C\,=\, T_C\otimes \pi_*{\mathcal O}_{\widetilde{C}}
\,=\, T_C\oplus (T_C\otimes \alpha).
$$
On the other hand, $\pi^*\alpha\, =\, {\mathcal O}_{\widetilde{C}}$ because $\widetilde{C}\, \subset\,
\alpha \setminus 0_C$, where $0_C$ is the image of the zero section of $\alpha$. So
$$
\pi^*(\pi_* T_{\widetilde{C}})\,=\, \pi^*T_C\oplus \pi^*(T_C\otimes \alpha) \,=\, T_{\widetilde{C}}\oplus
T_{\widetilde{C}}.
$$
Using this isomorphism we define
\begin{equation}\label{phi}
\phi\ :\ T_{\widetilde{C}}\ \longrightarrow\ \pi^*(\pi_* T_{\widetilde{C}})\ =\ T_{\widetilde{C}}\oplus
T_{\widetilde{C}}
\end{equation}
as the diagonal map $v\, \longmapsto\, (v,\, v)$. Let $w$ be
a local holomorphic coordinate function on a simply connected neighbourhood $U\, \subset\, \widetilde{C}$
of $y$ such that the restriction $\pi\big\vert_{U}\, :\, U\, \longrightarrow \, \pi(U)$ is a
biholomorphism. Then $z\, := \, w \circ (\pi\big\vert_{U})^{-1} $ is a
holomorphic coordinate function on the neighbourhood $\pi(U)$ of $x\, =\, \pi (y)$.
The map $\phi$ is: 
$$\left(\frac{\partial}{\partial w}\right)_y\ \longmapsto \
\left(\left(\frac{\partial}{\partial z}\right)_x,\, \left(\frac{\partial}{\partial z}\right)_x \otimes y\right).$$
As $\pi$ is finite, from \eqref{dis},
$$
H^1(\widetilde{C},\, {\mathcal O}_{\widetilde{C}})\,=\,
H^1(C,\, \pi_*{\mathcal O}_{\widetilde{C}})\,=\, H^1(C,\, {\mathcal O}_C)\oplus H^1(C,\, \alpha)
$$
Using this it is straightforward to check that
 \begin{equation}\label{inclusion}
{\widetilde{\psi}}_y\ =\ ( \psi_x,\, (\psi_P)_x) \circ \phi_y,
 \end{equation}
 where $\psi_P$ is the map in \eqref{a4} and $\psi_x$ (respectively, ${\widetilde{\psi}}_y$)
is the map in \eqref{cdC} (respectively, \eqref{cdtilde}).

 Let
 ${\mathbb C}_{C}$ (respectively, ${\mathbb C}_{\widetilde{C}}$) be the
constant sheaf on $C$ (respectively, $\widetilde{C}$) with stalk $\mathbb C$. Then
$\pi_*{\mathbb C}_{\widetilde{C}}\,=\, {\mathbb C}_{C}\oplus
{\mathbb C}_\alpha$. Using this and \eqref{dis} one checks that 
the map
\begin{equation}\label{i}
H^1(\widetilde{C},\,{\mathcal O}_{\widetilde{C}})\, \cong\, H^1(C,\, {\mathcal O}_C)
\oplus H^1(C,\, \alpha)\, \xrightarrow{\,\,(i_C,\, (i_P)_C)\,\,}\, 
H^1(C,\, {\mathbb C}) \oplus H^1(C,\, {\mathbb C}_{\alpha})\,\cong\,
H^1(\widetilde{C},\, {\mathbb C})
\end{equation}
coincides with $i_{\widetilde{C}}$ in \eqref{cdtilde}, where $(i_P)_C$, $i_C$ and
$i_{\widetilde{C}}$ are as in \eqref{cd}, \eqref{cdC} and \eqref{cdtilde} respectively. In other
words,
$$i_{\widetilde{C}}\,\ =\,\ (i_C,\, (i_P)_C).$$
{}From \eqref{inclusion} and \eqref{i} it follows that
\begin{equation}
\label{sopra}
i_{\widetilde{C}} \circ {\widetilde{\psi}}_y\ =\ (i_C,\, (i_P)_C) \circ (\psi_x,\, (\psi_P)_x) \circ \phi_y ,
\end{equation}
and from the diagrams \eqref{cd} and \eqref{cdC} we conclude that 
\begin{equation}
\label{sotto1}
(i_C,\, (i_P)_C) \circ (\psi_x, \,(\psi_P)_x)\circ \phi_y\ \, =\ \, (-j_x \circ \eta_x,
\, -(\chi_P)_x \circ (\eta_P)_x) \circ \phi_y.
\end{equation}
{}From diagram \eqref{cdtilde} it follows that
\begin{equation}
\label{sotto2}
i_{\widetilde{C}} \circ {\widetilde{\psi}_y}\ =\ -\,\widetilde{j}_y \circ\widetilde{\eta}_y.
\end{equation}

Let
\begin{equation}\label{si}
\sigma\ :\ \widetilde{C}\ \longrightarrow\ \widetilde{C}
\end{equation}
be the nontrivial deck transformation for the covering in \eqref{dc}. For
$y\, \in\, \widetilde C$, we have
$$K_{\widetilde C}(2y)\ \hookrightarrow\ K_{\widetilde C}(2y + 2\sigma(y))\ =\ \pi^*(K_C(2 x)),$$
where $x\,=\, \pi(y)$. Applying $\pi_*$, taking global sections and
using \eqref{dis} the following natural inclusion map is obtained:
\begin{equation}\label{ee}
\epsilon_y\, :\, H^0(\widetilde{C},\, K_{\widetilde C}(2y))\, \hookrightarrow\,
H^0(C,\, \pi_*\pi^* K_C(2x))\,=\, H^0(C,\, K_C(2x))\oplus H^0(C,\, (K_C\otimes \alpha)(2x)).
\end{equation}

\begin{lemma}
As the point $y$ moves over $\widetilde C$, the map $\epsilon_y$ varies holomorphically.
\end{lemma}

\begin{proof}
As before, denote $\widetilde{S} \,:=\, \widetilde{C} \times \widetilde{C}$;
let
\begin{equation}\label{tpq}
\widetilde{p},\ \widetilde{q}\ :\ \widetilde{S}\ \longrightarrow\ \widetilde{C}
\end{equation}
be the natural projections to the first and second factor respectively.
Let $$\Gamma\ =\ \{(x,\, \sigma(x))\,\in\, \widetilde{C} \times \widetilde{C}\,\, \big\vert\,\,\, x\, \in\,
\widetilde{C}\}\ \subset\ \widetilde{S}$$
be the graph of the map $\sigma$ in \eqref{si}. We have an inclusion map
\begin{equation}\label{e-epsilon}
\epsilon\,:\, \widetilde{V}\,:= \,\widetilde{p}_* ((\widetilde{q}^*K_{\widetilde{C}})
\otimes {\mathcal O}_{\widetilde{S}}(2 \widetilde{\Delta}) ) \,\hookrightarrow\,
\widetilde{W}\,:= \,\widetilde{p}_* (\widetilde{q}^*(K_{\widetilde{C}}) \otimes
{\mathcal O}_{\widetilde{S}}(2 \widetilde{\Delta} + 2 \Gamma)),
\end{equation}
where $\widetilde\Delta$, as before, is the reduced diagonal in $\widetilde{C}\times\widetilde{C}$. 
Again by Grauert's semicontinuity theorem, the direct image $\widetilde{W}$ in \eqref{e-epsilon} is a
holomorphic vector bundle over $\widetilde C$ of rank $2g+2$ with fiber
$$\widetilde W _y = H^0(\widetilde{C},\, K_{\widetilde{C}}(2y + 2\sigma(y)))
\ \cong\ H^0(C,\, K_C(2x)) \oplus H^0(C,\, (K_C \otimes \alpha)(2x)).$$
Furthermore, the map
$\epsilon_y$ in \eqref{ee} is induced by the inclusion map of vector bundles $\epsilon$
in \eqref{e-epsilon}. Consequently, $\epsilon_y$ varies holomorphically as $y$ moves over $\widetilde C$.
\end{proof}

\begin{proposition}\label{epsilon}
There is a natural isomorphism of vector bundles: 
\begin{equation}\label{iso}
\widetilde{W}\ \cong\ \pi^*V \oplus \pi^*V_P,
\end{equation}
where $V\,=\,p_*((q^*K_C) ( 2\Delta))$ and $V_P\, = \,p_*(q^*(K_C \otimes \alpha) ( 2\Delta))$
(see \eqref{VP}).
\end{proposition}

\begin{proof}
Set 
$$\Phi\, =\, (\pi,\, \pi)\,:\, \widetilde{S} \,=\, \widetilde{C}\times\widetilde{C}\,
\longrightarrow\, C \times C\, =\, S.$$
Then, clearly
$$\Phi^*({\mathcal O}_{S}({\Delta}))\ =\ {\mathcal O}_{\widetilde S}(\widetilde{\Delta}+ \Gamma) .$$
We have 
$$\pi \circ \widetilde{p}\ =\ p \circ \Phi,\ \ \ \pi \circ \widetilde{q}\ =\ q \circ \Phi.$$
Consider the diagram
\begin{equation}\label{ed}
\xymatrix@R=0.8cm@C=0.8cm{
\widetilde{C} \times \widetilde{C} \ar[rr]^{ f } \ar@/^3pc/[rrr] ^{\Phi} \ar@/_1pc/[drr]_{\widetilde{p}}
 & &
 \widetilde{C} \times C \cong X \ar[d]^{p_1} \ar[r]^{\qquad {\textbf g}\,} 
 & C \times C \ar[d]^{p } 
\\ & & \widetilde{C} \ar[r]^{\pi } &C& }
\end{equation}
where $f\,=\, {\rm Id}_{\widetilde{C}}\times\pi$,
${\textbf g}\,=\, \pi\times{\rm Id}_{C}$, and $p_1,\, p$ are
projections to the first factor; see \eqref{tpq} for
$\widetilde{p}$. So the square diagram on the right of \eqref{ed} is
Cartesian. Evidently, we have ${\textbf g} \circ f \,=\, \Phi$ and
$p_1 \circ f \,=\, \widetilde{p}$. There are the isomorphisms
$$\widetilde{W} \,=\, \widetilde{p}_* ((\widetilde{q}^*K_{\widetilde{C}}) \otimes
{\mathcal O}_{\widetilde{S}}(2 \widetilde{\Delta} + 2 \Gamma))
\,\cong\, \widetilde{p}_* (\widetilde{q}^*(\pi^*K_C) \otimes \Phi^*{\mathcal O}_{S}(2 \Delta))$$
$$= \ \widetilde{p}_* (\Phi^*((q^*K_C) \otimes {\mathcal O}_{S}(2 \Delta))),$$
because $\pi \circ \widetilde{q} \,=\, q \circ \Phi$. So we have 
$$\widetilde{W}\,=\, \widetilde{p}_* (\Phi^*((q^*K_C) \otimes {\mathcal O}_{S}(2 \Delta)))
\,=\, (p_1)_* \circ f_* (({\textbf g}\circ f)^* ( (q^*K_C) \otimes
{\mathcal O}_{S}(2 \Delta))),$$ because
$\widetilde{p}\, =\,p_1 \circ f$ and $\Phi\,=\, {\textbf g} \circ
f$. Denote
${\mathcal G}\, :=\, (q^*K_C) \otimes {\mathcal O}_{S}(2
\Delta)$. Hence it follows that
$$\widetilde{W}\ =\ (p_1)_* f_* (f^* {\textbf g}^* {\mathcal G})\ =\
 (p_1)_*(({\textbf g}^*{\mathcal G}) \otimes f_* {\mathcal O}_{\widetilde S}).$$

Note that $f_* {\mathcal O}_{\widetilde S}\, =\, ( Id_{\widetilde{C}} \times \pi)_*(
{\mathcal O}_{\widetilde C}\boxtimes{\mathcal O}_{\widetilde C})\,=\,
{\mathcal O}_{\widetilde C} \boxtimes \pi_* {\mathcal O}_{\widetilde C}
\,=\, {\mathcal O}_{\widetilde C} \boxtimes ({\mathcal O}_C \oplus \alpha),$ so
$$\widetilde{W}\ =\ (p_1)_*(({\textbf g}^*{\mathcal G}) \otimes f_* {\mathcal O}_{\widetilde S})
\ =\ (p_1)_*(({\textbf g}^*{\mathcal G})) \oplus (p_1)_*(({\textbf g}^*{\mathcal G}) \otimes
({\mathcal O}_{\widetilde{C}} \boxtimes \alpha))$$
$$=\ (p_1)_*(({\textbf g}^*{\mathcal G})) \oplus (p_1)_*(({\textbf g}^*{\mathcal G}) \otimes
{\textbf g}^*({\mathcal O}_{C} \boxtimes \alpha)), $$
where the last equality follows from the isomorphisms
$${\mathcal O}_{\widetilde{C}} \boxtimes \alpha\ =\ \pi^*{\mathcal O}_C \boxtimes{\rm Id}_C^*
(\alpha)\ =\ {\textbf g}^*({\mathcal O}_C \boxtimes \alpha).$$ 
Since the square on the right of \eqref{ed} is a Cartesian diagram, we
can apply flat base change \cite[Prop. 9.3 p. 255]{hartshorne}:
$$(p_1)_* \circ {\textbf g}^*\ =\ \pi^* \circ p_*,$$
so
\begin{gather*}
(p_1)_*({\textbf g}^*{\mathcal G})\ =\ \pi^* p_*{\mathcal G}\ =\ \pi^* V,\\
(p_1)_*(({\textbf g}^*{\mathcal G}) \otimes{\textbf g}^*({\mathcal O}_{C} \boxtimes \alpha))
\ =\ \pi^* p_*( q^*K_C \otimes {\mathcal O}_S(2\Delta) \otimes q^*\alpha)\ =\ \pi^* V_P, 
\end{gather*}
and therefore $\widetilde{W}\, \cong\, \pi^* V \oplus \pi^* V_P$.
\end{proof}

\begin{proposition}
\label{jtilde}
The homomorphism $$\widetilde{j}\ :\ \widetilde{p}_*(\widetilde{q}^*K_{\widetilde C}(2\widetilde{\Delta}))
\ \longrightarrow\ H^1(\widetilde{C},\, {\mathbb C}) \otimes_{\mathbb C} {\mathcal O}_{\widetilde{C}}$$
(see \eqref{cdtilde}) extends to a homomorphism
$$\widetilde{J}\ :\ {\mathcal N}\ \longrightarrow\ H^1(\widetilde{C},\,
{\mathbb C}) \otimes_{\mathbb C} {\mathcal O}_{\widetilde{C}},$$
where ${\mathcal N}$ is a sub-vector bundle of $\widetilde{W}$.
Under the isomorphism \eqref{iso} the vector bundle $\mathcal N$ splits as 
 $\pi^*V \oplus \widetilde{\mathcal H}$
where $\widetilde{\mathcal H}\,\subset\, \pi^*V_P$ is a sub-vector bundle of rank $g$.
Moreover, under these isomorphisms,
$$\widetilde{J}\ =\ (\pi^*(j),\, \widetilde{\chi}_P)\ :\ {\mathcal N}\ \cong\
\pi^*V \oplus \widetilde{\mathcal H}\ \longrightarrow\ 
H^1(\widetilde{C},\, {\mathbb C}) \otimes_{\mathbb C} {\mathcal O}_{\widetilde{C}}
$$
$$
\cong\ (H^1(C,\, {\mathbb C}) \otimes_{\mathbb C} {\mathcal O}_{C}) \oplus
(H^1(C,\, {\mathbb C}_{\alpha}) \otimes_{\mathbb C} {\mathcal O}_{C}),
$$
where $\widetilde{\chi}_P\,:\, \widetilde{\mathcal H}\, \longrightarrow\,
 H^1(C,\, {\mathbb C}_{\alpha}) \otimes_{\mathbb C} {\mathcal O}_{C}$.
\end{proposition}

\begin{proof}
We follow the strategy of the proof of \cite[Proposition 3.5]{bcfp}.
Consider the de Rham differential 
$$
d\,:\, {\mathcal O}_{\widetilde{S}}(\widetilde{\Delta} + \Gamma)\,\longrightarrow\,
\Omega^1_{\widetilde{S}} \otimes{\mathcal O}_{\widetilde{S}}(\widetilde{2\Delta} +
2\Gamma)\, \cong\, (\widetilde{p}^* K_{\widetilde{C}} \oplus \widetilde{q}^*
K_{\widetilde{C}}) \otimes{\mathcal O}_{\widetilde{S}}(\widetilde{2\Delta} + 2\Gamma),
$$
and denote by
\begin{equation}\label{et}
T\, :\,{\mathcal O}_{\widetilde{S}}(\widetilde{\Delta} + \Gamma)
\, \longrightarrow\,
\widetilde{q}^* K_{\widetilde{C}} \otimes{\mathcal O}_{\widetilde{S}}(\widetilde{2\Delta}
+ 2\Gamma)
\end{equation}
the composition of the above homomorphism $d$ with the projection of
$(\widetilde{p}^* K_{\widetilde{C}} \oplus \widetilde{q}^*
K_{\widetilde{C}}) \otimes{\mathcal O}_{\widetilde{S}}(\widetilde{2\Delta} + 2\Gamma)$
onto the second factor.

The following lemma will be used in the proof of the proposition.

\begin{lemma}\label{lem-p}
The kernel of $T$ in \eqref{et} is the subsheaf\,
$p^{-1}({\mathcal O}_{\widetilde{C}})$ of
${\mathcal O}_{\widetilde{S}}(\widetilde{\Delta} + \Gamma)$.
\end{lemma}

\begin{proof}[{Proof of Lemma \ref{lem-p}}]
The result is local in $\widetilde{C}$.

Fix a holomorphic local coordinate function $z$ on $\widetilde{C}$, and set $z_1\,:=\, z \circ
\widetilde{p}$ and $z_2\,:=\, z \circ\widetilde{q}$. Since $\widetilde{\Delta}\cap
\Gamma \,=\, \emptyset$, we can 
choose an open covering of $\widetilde{C}\times\widetilde{C}$ by open subsets
of the form $\widetilde{U}\, =\, U_1 \times U_2$, with coordinate function
$(z_1,\, z_2)$ such that the following conditions hold:
\begin{itemize}
\item If $\widetilde{U}\cap\widetilde{\Delta}\, \neq\, \emptyset$, then
$\widetilde{U} \cap \Gamma \,=\, \emptyset$, and

\item if $\widetilde{U} \cap \Gamma \,\neq\, 
\emptyset$, then $\widetilde{U}\cap\widetilde{\Delta}\, =\, \emptyset$.
\end{itemize}

First assume that $\widetilde{U} \cap \Gamma\, =\, \emptyset$. Set $x \,=\, z_2-z_1$,
\, $y \,=\, z_1 + z_2$, and take a local section $f\,=\, \frac{a(y)}{x} + b(x,y)$
of ${\mathcal O}_{\widetilde{S}}(\widetilde{\Delta} +
\Gamma)\big\vert_{\widetilde{U}}\, =\, {\mathcal O}_{\widetilde{S}}
(\widetilde{\Delta})\big\vert_{\widetilde U}$, where $b(x,\,y)$ is holomorphic. We have 
$$T(f)\ =\ \frac{\partial f}{\partial z_2} dz_2\ =\
\left(-\frac{a(y)}{x^2} + \frac{a'(y)}{x}\right) dz_2 + \frac{\partial b}{\partial z_2}dz_2,$$
and hence $T(f) \,=\, 0$ if and only if the following two statements hold:
\begin{itemize}
\item $a(y)\,\equiv\, 0$, and
\item $\frac{\partial b}{\partial z_2} \,=\,0$.
\end{itemize}
Therefore, $T(f)\,=\, 0$ if and only if $f \,=\, b(z_1)$, which proves
the lemma in this case.

Next assume that $\widetilde{U}\cap\widetilde{\Delta}\, =\, \emptyset$. Set
$w\,=\, z_2 - \sigma(z_1)$ and $u \,=\, z_2 + \sigma(z_1)$. Take a local section
$f\,=\, \frac{a(u)}{w} + b(u,w)$ of ${\mathcal O}_{\widetilde{S}}(\widetilde{\Delta} +
\Gamma)\big\vert_{\widetilde{U}} \,=\, {\mathcal O}_{\widetilde{S}}
(\Gamma)\big\vert_{\widetilde{U}}$, where $b(u,\, w)$ is holomorphic. We have 
$$T(f)\ =\ \frac{\partial f}{\partial z_2} dz_2\ =\
\left(-\frac{a(u)}{w^2} + \frac{a'(u)}{w}\right) dz_2 +
\frac{\partial b}{\partial z_2}dz_2,$$ hence
$T(f) \,=\, 0$ if and only if the following two hold:
\begin{itemize}
\item $a(u)\,\equiv\, 0$, and
\item $\frac{\partial b}{\partial z_2} \,=\,0$.
\end{itemize}
Therefore, $T(f)\,=\, 0$ if and only if $f \,=\, b(z_1)$, which proves the lemma in this
case also. Thus the proof of the lemma is complete.
\end{proof}

Continuing with the proof of the proposition,
in view of Lemma \ref{lem-p}, setting
$\widetilde{\mathcal E}\, =\, {\rm Im}(T)$, there is the following short exact sequence of sheaves
on $\widetilde{S}$
\begin{equation}\label{1}
0\, \longrightarrow\, {\widetilde{p}}^{-1}({\mathcal O}_{\widetilde{C}})\, \longrightarrow\,
{\mathcal O}_{\widetilde{S}}(\widetilde{\Delta} + \Gamma)\, \longrightarrow\,
\widetilde{\mathcal E}\, \longrightarrow\, 0.
\end{equation}

We will construct a map
\begin{equation}\label{r}
r\ =\ (r_1,\, r_2)\ :\ \widetilde{q}^* K_{\widetilde{C}} \otimes
{\mathcal O}_{\widetilde{S}}(\widetilde{2\Delta} + 2\Gamma)\ \longrightarrow\
{\widetilde{p}}^{-1}({\mathcal O}_{\widetilde{C}}) \oplus
{\widetilde{p}}^{-1}({\mathcal O}_{\widetilde{C}})
\end{equation}
as follows: Fix an open covering of $\widetilde{S}$ by open
subsets of the form $\widetilde{U}$ such that 
\begin{itemize}
\item either $\widetilde{U} \cap (\widetilde{\Delta} \cup \Gamma)\, =\, \emptyset$,
 
\item or $\widetilde{U}\cap\widetilde{\Delta}\, \neq\,\emptyset$ and
$\widetilde{U} \cap \Gamma \,=\, \emptyset$,

\item or $\widetilde{U} \cap \widetilde{\Delta}\, =\,\emptyset$ and
$\widetilde{U} \cap \Gamma \,\neq\, \emptyset$.
\end{itemize}
In the first case where $\widetilde{U}\cap (\widetilde{\Delta} \cup \Gamma) \,=\, \emptyset$, for a
section $\omega$ of $\widetilde{q}^* K_{\widetilde{C}} \otimes
{\mathcal O}_{\widetilde{S}}(\widetilde{2\Delta} + 2\Gamma)\big\vert_{\widetilde{U}}$,
define $r_1(\omega)\, =\,0$ and $r_2(\omega) \,=\,0$. In the
second case where $\widetilde{U} \cap \widetilde{\Delta}
\,\neq \, \emptyset$ and $\widetilde{U} \cap \Gamma \,=\, \emptyset$, define $r_2(\omega)
\,=\,0$ and define
$$r_1(\omega)\ :=\ \int_{\gamma_{z_1}} \omega,$$
where $(z_1,\,z_2) \,\in \,\widetilde{U}\, =\, U_1 \times U_2$ and $\gamma_{z_1}$ is a
small oriented circle around $z_1$ in $(\{z_1\}\times\widetilde{C})\cap\widetilde{U}$; 
assuming $D_{2 \epsilon}\times D_{2 \epsilon}\,\subset\, U_1 \times U_2$, and
taking $\omega\, =\, (\frac{A(y)}{x^2} + \frac{B(y)}{x} + C(x,y))dz_2$, we have 
$$r_1(\omega)(z_1)\ =\ \int_{|z_2-z_1| = \epsilon} \omega = \int_{|z_2-z_1|
= \epsilon} \left(\frac{A(2z_1) + (z_2-z_1)(A'(2z_1) + B(2z_1))}{(z_2-z_1)^2} + h\right) dz_2$$
$$=\ 2 \pi\sqrt{-1} (A'(2z_1) + B(2z_1)),$$
where $h$ is holomorphic. So in this case $r(\omega)\, =\, (r_1(\omega),\, r_2(\omega) )
\,=\, (0,\,0)$ if and only if $A'(2z_1) \,=\, - B(2z_1)$, so $r(\omega)\, =\, 0$ if and only if
$\omega \,\in\, {\widetilde{\mathcal E}}(\widetilde{U})$. In the third
case where $\widetilde{U} \cap\widetilde{\Delta}\,= \, \emptyset$ and $\widetilde{U}\cap
\Gamma \,\neq\, \emptyset$, set $r_1(\omega) \,=\,0$ and define
$$r_2(\omega)\ :=\ \int_{\gamma_{\sigma(z_1)}} \omega,$$
where $\gamma_{\sigma(z_1)}$ is a small oriented circle around $\sigma(z_1)$ in
$(\{\sigma(z_1)\} \times \widetilde{C})\cap\widetilde{U}$. 
Assuming that $D_{2 \epsilon} \times D_{2 \epsilon}\,\subset\, U_1 \times U_2$ and
taking $\omega \,=\, (\frac{A(u)}{w^2} + \frac{B(u)}{w} + C(u,w))dz_2$, we have 
$$r_2(\omega)(z_1)\ =\ \int_{|z_2-\sigma(z_1)| = \epsilon} \omega$$
$$=\,\int_{|z_2-\sigma(z_1)| = \epsilon} \left(\frac{A(2\sigma(z_1)) + (z_2-\sigma(z_1))(A'(2\sigma(z_1)) + B(2\sigma(z_1)))}{(z_2-\sigma(z_1))^2} + h\right) dz_2$$
$$=\ 2 \pi\sqrt{-1} (A'(2\sigma(z_1)) + B(2\sigma(z_1))).$$
So, $r(\omega) \,=\, (r_1(\omega),\, r_2(\omega) ) \,=\, (0,\,0)$ if and only if
$A'(2\sigma(z_1)) \,=\, - B(2\sigma(z_1))$, so $r(\omega) \,=\, 0$ if and only if
$\omega \,\in\, {\widetilde{\mathcal E}}(\widetilde{U})$. 

Thus, ${\widetilde{\mathcal E}}$ is the kernel of $r\,=\,(r_1,\,r_2)$. Hence there is
the following short exact sequence of sheaves on $\widetilde{S}$
\begin{equation}\label{2}
0\, \longrightarrow\,\widetilde{\mathcal E}\, \longrightarrow\,
\widetilde{q}^* K_{\widetilde{C}} \otimes{\mathcal O}_{\widetilde{S}}(\widetilde{2\Delta}
+ 2 \Gamma)\, \longrightarrow\, {\widetilde{p}}^{-1}({\mathcal O}_{\widetilde{C}})
\oplus {\widetilde{p}}^{-1}({\mathcal O}_{\widetilde{C}})\, \longrightarrow\, 0.
\end{equation}

Denote by $s_1$ the restriction of $r_1$ (see \eqref{r})
to $\widetilde{q}^* K_{\widetilde{C}} \otimes {\mathcal 
O}_{\widetilde{S}}(\widetilde{2\Delta})$, and also denote by $s_2$ the restriction of $r_2$
to $\widetilde{q}^* K_{\widetilde{C}} \otimes {\mathcal O}_{\widetilde{S}}(2\Gamma)$.
Set $${\mathcal E}_1\ :=\ Ker (\widetilde{q}^* K_{\widetilde{C}}\otimes
{\mathcal O}_{\widetilde{S}}(\widetilde{2\Delta})\ \stackrel{s_1}\longrightarrow\
{\widetilde{p}}^{-1}({\mathcal O}_{\widetilde{C}})),$$
$${\mathcal E}_2\ :=\ Ker (\widetilde{q}^* K_{\widetilde{C}} \otimes
{\mathcal O}_{\widetilde{S}}(2\Gamma)\ \stackrel{s_2}\longrightarrow\
{\widetilde{p}}^{-1}({\mathcal O}_{\widetilde{C}})).$$
They fit in the short exact sequences
\begin{equation}\label{E_1}
0 \,\longrightarrow\, {\mathcal E}_1\, \longrightarrow\, \widetilde{q}^* K_{\widetilde{C}}
\otimes{\mathcal O}_{\widetilde{S}}(\widetilde{2\Delta})\, \longrightarrow\,
{\widetilde{p}}^{-1}({\mathcal O}_{\widetilde{C}})\, \longrightarrow\, 0,
\end{equation}
\begin{equation}
\label{E_2}
0\, \longrightarrow\, {\mathcal E}_2\,\longrightarrow\, \widetilde{q}^* K_{\widetilde{C}}
\otimes{\mathcal O}_{\widetilde{S}}(2\Gamma)\,\longrightarrow\,
{\widetilde{p}}^{-1}({\mathcal O}_{\widetilde{C}})\, \longrightarrow\, 0.
\end{equation}

Clearly we have the inclusion maps ${\mathcal E}_i \,\subset\, {\widetilde{\mathcal E}}
\,=\, Ker(r_1,\,r_2)$, $i\, =\,1,\,2$. 

Taking the push-forward, by $\widetilde p$, of \eqref{2} the following
sequence is obtained
\begin{equation}\label{p2}
0\,\longrightarrow\, \widetilde{p}_*(\widetilde{\mathcal E})\, \longrightarrow\,
\widetilde{W}\, \xrightarrow{\,\,\,(\overline{r}_1, \overline{r}_2)\,\,\,}\,
{\mathcal O}_{\widetilde{C}}\oplus {\mathcal O}_{\widetilde{C}}\, \longrightarrow\,
R^1{\widetilde{p}}_*(\widetilde{\mathcal E}) \, \longrightarrow\, 0,
\end{equation}
where $(\overline{r}_1,\, \overline{r}_2)$ is the map induced by $(r_1,\,r_2)$ on the
push-forward $\widetilde{W}$ (of $\widetilde{q}^* K_{\widetilde{C}}\otimes 
{\mathcal O}_{\widetilde{S}}(\widetilde{2\Delta} + 2 \Gamma)$). Moreover,
we have $\widetilde{p}_*({\mathcal E}_1) \,\cong\, \widetilde{V}$ (see \cite[(3.11)]{bcfp}),
and $$\widetilde{p}_*({\mathcal E}_2)\ \cong\ \widetilde{p}_*(\widetilde{q}^* K_{\widetilde{C}}
\otimes{\mathcal O}_{\widetilde{S}}(2\Gamma)).$$ The last isomorphism is analogous to the
previous one and it can be seen as follows: 
Take a small open subset $\widetilde{U}\, \subset\, \widetilde{C}$ as before that
intersects $\Gamma$ and does not intersect $\widetilde{ \Delta}$, and
take $$\omega \,\in\, H^0(\widetilde{U},\, \widetilde{p}_*(\widetilde{q}^* K_{\widetilde{C}}
\otimes{\mathcal O}_{\widetilde{S}}(2\Gamma))) \,=\, H^0(\widetilde{U} \times \widetilde{C},
\, \widetilde{q}^* K_{\widetilde{C}} \otimes{\mathcal O}_{\widetilde{S}}(2\Gamma)).$$
Then $\omega(y, \,\cdot) \,\in\, H^0(\widetilde{C},\, K_{\widetilde{C}}(2 \sigma(y)))$
for any $y \,\in\, \widetilde{U}$, so it follows that $r_2(\omega)(\sigma(y))\, =\,0$.
Consequently, $\omega \,\in\, H^0(\widetilde{U} \times \widetilde{C},\,
\widetilde{\mathcal E}) \,=\, \widetilde{p}_*(\widetilde{\mathcal E})(\widetilde{U})$. 

Hence we have the inclusion maps
$$\widetilde{V}\ \cong\ \widetilde{p}_*({\mathcal E}_1)\ \subset\
\widetilde{p}_*(\widetilde{\mathcal E})$$ and
$$\widetilde{{\mathcal U}}\ :=\ \widetilde{p}_*(\widetilde{q}^* K_{\widetilde{C}}
\otimes{\mathcal O}_{\widetilde{S}}(2\Gamma))\ \cong\
\widetilde{p}_*({\mathcal E}_2)\ \subset\ \widetilde{p}_*(\widetilde{\mathcal E}).$$

At a point $y\, \in \,\widetilde{C}$, the map $(\overline{r}_1,\, \overline{r}_2)$
in \eqref{p2}
$$ \widetilde{W}_y\ =\ H^0(\widetilde{C},\, K_{\widetilde{C}}(2y + 2 \sigma(y))\
\xrightarrow{\,\,\,(\overline{r}_1, \overline{r}_2)_y\,\,\,}\ {\mathbb C} \oplus {\mathbb C} $$
is given by $(\overline{r}_1,\, \overline{r}_2)_y(\omega)\,=\, (res_y(\omega),
\, res_{\sigma(y)}(\omega))$. 
Hence we have the exact sequence
\begin{equation}
\label{2y}
0\, \rightarrow\, H^0(\widetilde{C},\, K_{\widetilde{C}})\, \longrightarrow\,
H^0(\widetilde{C},\, K_{\widetilde{C}}(2y)) \oplus H^0(\widetilde{C},\, K_{\widetilde{C}}
(2 \sigma(y))\, \longrightarrow
\end{equation}
$$
\longrightarrow\, H^0(\widetilde{C},\, K_{\widetilde{C}}(2y + 2 \sigma(y)))\,
\xrightarrow{\,\,\,(\overline{r}_1, \overline{r}_2)_y\,\,\,}\,{\mathbb C}
\oplus {\mathbb C}\, \stackrel{s}{\longrightarrow}\, {\mathbb C}\,\longrightarrow\, 0,
$$
where $s(a,\,b) \,=\, a+b$, and exact sequence \eqref{p2} at $y$ is 
$$0\ \longrightarrow\ (H^0(\widetilde{C},\, K_{\widetilde{C}}(2y)) \oplus
H^0(\widetilde{C},\, K_{\widetilde{C}}(2 \sigma(y)))/H^0(\widetilde{C},\, K_{\widetilde{C}})
\ \longrightarrow\ H^0(K_{\widetilde{C}}(2y + 2 \sigma(y)))
$$
$$
\xrightarrow{\,\,\,(\overline{r}_1, \overline{r}_2)_y\,\,\,}\ {\mathbb C}
\oplus {\mathbb C}\ \stackrel{s}{\longrightarrow}\ {\mathbb C}\ \longrightarrow\ 0.$$

Since $\widetilde{V}_y\, \cong\, \widetilde{p}_*({\mathcal E}_1)_y\, =\, H^0(\widetilde{C},
\, K_{\widetilde{C}}(2y))$ and $\widetilde{{\mathcal U}}_y\,=\, \widetilde{p}_*({\mathcal E}_2)
\,=\, H^0(\widetilde{C},\, K_{\widetilde{C}}(2 \sigma(y))$, the map
$(\overline{r}_1,\, \overline{r}_2)$ in \eqref{p2} fits in the following exact sequence:
\begin{equation}\label{p2'}
0\,\longrightarrow\, \widetilde{p}_*(\widetilde{q}^* K_{\widetilde{C}})\, \longrightarrow\,
\widetilde{V} \oplus \widetilde{\mathcal U}\, \longrightarrow\,
\widetilde{W}\, \xrightarrow{\,\,\,(\overline{r}_1, \overline{r}_2)\,\,\,}\,
{\mathcal O}_{\widetilde{C}}\oplus {\mathcal O}_{\widetilde{C}}\, \longrightarrow
\, {\mathcal O}_{\widetilde{C}}\, \longrightarrow\, 0,
\end{equation}
so we have isomorphisms: 
$$\widetilde{p}_*(\widetilde{\mathcal E}) \,\cong\,(\widetilde{V} \oplus \widetilde{\mathcal U})/
\widetilde{p}_*(\widetilde{q}^* K_{\widetilde{C}}),\ \ \ 
R^1{\widetilde{p}}_*(\widetilde{\mathcal E})\,\cong\, {\mathcal O}_{\widetilde{C}}.$$
 
Now, with the notation of the proof of Proposition \ref{epsilon}, we have, for $i\, =\,0,\,1$,
the following:
$$R^i\widetilde{p}_*{\mathcal O}_{\widetilde{S}}(\widetilde{\Delta} + \Gamma)
\,\cong\, R^i\widetilde{p}_*(\Phi^*({\mathcal O}_{\widetilde{S}}(\Delta)))\,=\,
R^i(p_1 \circ f)_*((g \circ f)^*({\mathcal O}_{S}(\Delta))$$
$$\cong\, (R^i(p_1)_* \circ f_*)(f^* \circ g^*({\mathcal O}_{S}(\Delta)))
\,\cong\, R^i(p_1)_*(g^*({\mathcal O}_{S}(\Delta)) \otimes ({\mathcal O}_{\widetilde{C}}
\boxtimes ({\mathcal O}_C \oplus \alpha)))$$
$$\cong\, R^i(p_1)_*(g^*({\mathcal O}_{S}(\Delta)) \oplus R^1(p_1)_*(g^*({\mathcal O}_{S}(\Delta) \otimes q^*(\alpha))).$$
Finally, using flat base change (as done in the proof of Proposition \ref{epsilon}),
it follows that 
\begin{equation}
\label{bb}
R^i\widetilde{p}_*{\mathcal O}_{\widetilde{S}}(\widetilde{\Delta} + \Gamma)\ \cong\
\pi^*(R^ip_*({\mathcal O}_{S}(\Delta)) \oplus \pi^*(R^ip_*({\mathcal O}_{S}(\Delta)
\otimes q^*(\alpha)))
\end{equation}
for $i\,=\,0,\,1$.

So 
$$\widetilde{p}_*{\mathcal O}_{\widetilde{S}}(\widetilde{\Delta} + \Gamma)\ \cong\
{\mathcal O}_{\widetilde{C}} \oplus \pi^*(p_*({\mathcal O}_S(\Delta) \otimes q^*(\alpha))), $$
from which it follows that if $C$ is not hyperelliptic, then
$p_*({\mathcal O}_S(\Delta) \otimes q^*(\alpha))\, =\,0$, because $H^0(C,\, \alpha(x))\,=\, 0$,
for all $x \,\in\, C$. Consequently, we have $\widetilde{p}_*{\mathcal O}_{\widetilde{S}}
(\widetilde{\Delta} + \Gamma) \,\cong\, {\mathcal O}_{\widetilde{C}}
\,\cong\, \widetilde{p}_* (\widetilde{p}^{-1}( {\mathcal O}_{\widetilde{C}}))$. 

In \cite[(3.12), (3.13)]{bcfp} it was shown that
$$R^1\widetilde{p}_* (\widetilde{p}^{-1}( {\mathcal O}_{\widetilde{C}}))\, \cong\,
H^1(\widetilde{C},\, {\mathbb C}) \otimes {\mathcal O}_{\widetilde{C}},\ \ \
R^2\widetilde{p}_* (\widetilde{p}^{-1}( {\mathcal O}_{\widetilde{C}}))
\,\cong\, {\mathcal O}_{\widetilde{C}}.$$ 
So, applying $\widetilde{p}_*$ to the exact sequence in \eqref{1} the following exact
sequence is obtained:
\begin{equation}
 \label{p1}
 0\,\longrightarrow\, {\mathcal O}_{\widetilde{C}}\,\longrightarrow\,
\widetilde{p}_*{\mathcal O}_{\widetilde{S}}(\widetilde{\Delta} + \Gamma)\, \rightarrow\,
\widetilde{p}_*(\widetilde{\mathcal E})\, \stackrel{\widetilde{J}}\longrightarrow\,
H^1(\widetilde{C},\, {\mathbb C}) \otimes {\mathcal O}_{\widetilde{C}}\,\longrightarrow\,
R^1\widetilde{p}_*{\mathcal O}_{\widetilde{S}}(\widetilde{\Delta} + \Gamma)\,
\longrightarrow\, 0,
 \end{equation}
and this gives the extension ${\widetilde{J}}$ of the map $\widetilde{j}$ to the vector bundle
$$ {\mathcal N}\ :=\ \widetilde{p}_*(\widetilde{\mathcal E})\ \subset\ \widetilde{W}.$$ 
 
Notice that if $C$ is not hyperelliptic, then the map $\widetilde{J}$ is injective. 

Using \eqref{bb}, equation \eqref{p1} becomes 
$$
0\,\longrightarrow\, \pi^*(p_*({\mathcal O}_S(\Delta) \otimes q^*(\alpha)) )\,\longrightarrow\,
\widetilde{p}_*(\widetilde{\mathcal E})\, \stackrel{\widetilde{J}}{\longrightarrow}\,
H^1(\widetilde{C},\, {\mathbb C}) \otimes {\mathcal O}_{\widetilde{C}}$$
$$\longrightarrow\, \pi^*(R^1p_*({\mathcal O}_{S}(\Delta)) \oplus \pi^*(R^1p_*({\mathcal O}_{S}(\Delta)
\otimes q^*(\alpha)))\, \longrightarrow\, 0
$$
At the end of the proof of \cite[Proposition 3.5]{bcfp} it is shown that there is a short exact sequence 
$$0\, \longrightarrow\, V \,\stackrel{j}{\longrightarrow}\, H^1(C,\, {\mathbb C}) \otimes {\mathcal O}_C
\,\longrightarrow\, R^1p_*{\mathcal O}_S(\Delta)\, \longrightarrow\, 0.$$
Taking its pullback by the map $\pi$ yields the short exact sequence 
\begin{equation}
\label{j1}
0\, \longrightarrow\, \pi^*V\, \stackrel{\pi^*(j)}{\longrightarrow}\, H^1(C,\, {\mathbb C})
\otimes {\mathcal O}_{\widetilde{C}}\, \longrightarrow\, \pi^* (R^1p_*{\mathcal O}_S(\Delta))
\, \longrightarrow\, 0.
\end{equation}

The inclusion map ${\mathcal N} \,\subset\, \widetilde{W} \,\cong\, \pi^*V \oplus \pi^* V_P$, the
isomorphism $$H^1(\widetilde{C}, \,{\mathbb C}) \otimes {\mathcal O}_{\widetilde{C}}
\,\cong\, (H^1(C,\, {\mathbb C}) \otimes {\mathcal O}_{\widetilde{C}}) \oplus
(H^1(C,\, {\mathbb C}_{\alpha}) \otimes {\mathcal O}_{\widetilde{C}})$$
and the isomorphism in \eqref{bb} together imply that the direct image,
under the map $\pi$, of the exact sequence in
\eqref{p1} splits in two exact sequences: the first one is \eqref{j1}, while the second one is 
\begin{equation}
\label{H}
0\, \longrightarrow \,\pi^*(p_*({\mathcal O}_S(\Delta) \otimes q^*(\alpha)))\, \longrightarrow\,
\widetilde{{\mathcal H} }\, \xrightarrow{\,\,\,\widetilde{\chi}_P\,\,\,}\,
 H^1(C,\, {\mathbb C}_{\alpha}) \otimes {\mathcal O}_{\widetilde{C}}
\end{equation}
$$
\longrightarrow \, \pi^* (R^1p_*({\mathcal O}_S(\Delta) \otimes q^*(\alpha)))\, \longrightarrow\, 0,
$$
where $\widetilde{{\mathcal H}}\, \subset\, \pi^*(V_P)$ and ${\mathcal N} \,\cong\,
\pi^*V \oplus\widetilde{\mathcal H}$. 

When $C$ is not hyperelliptic, the following is a short exact sequence
\begin{equation}\label{H'}
0\, \longrightarrow\, \widetilde{{\mathcal H} }\,\xrightarrow{\,\,\,\widetilde{\chi}_P\,\,\,}\,
H^1(C,\, {\mathbb C}_{\alpha}) \otimes {\mathcal O}_{\widetilde{C}}\, \longrightarrow\,
\pi^* (R^1p_*({\mathcal O}_S(\Delta) \otimes q^*(\alpha)))\,\longrightarrow\, 0.
\end{equation}
This completes the proof.
\end{proof}

\begin{remark}\label{rem}
Taking the push-forward of the exact sequence in \eqref{H} by $\pi$ and then taking its invariant part
(for the action of ${\mathbb Z}/2\mathbb Z$), we have the exact sequence
\begin{equation}\label{j_P}
0\, \longrightarrow\, p_*({\mathcal O}_S(\Delta) \otimes q^*(\alpha))\, \longrightarrow\,
{\mathcal H}\, \xrightarrow{\,\,\,\chi_P\,\,\,}\, H^1(C,\, {\mathbb C}_{\alpha}) \otimes
{\mathcal O}_{C}
\end{equation}
$$
\, \longrightarrow \, R^1p_*({\mathcal O}_S(\Delta) \otimes q^*(\alpha))\, \longrightarrow\, 0,
$$
where ${\mathcal H}\, =\, \pi_*(\widetilde{\mathcal H})^+ \,\subset\, V_P$ (the invariant part
under $\sigma$). Both ${\mathcal H}_x
\,=\, H^0(C,\, K_C \otimes \alpha) \oplus \langle (s_P)_x \rangle$ and $(\chi_P)_x$ are as in
the diagram in \eqref{diagx}. Hence ${\mathcal H}$ is a holomorphic vector bundle of rank $g$ and $\widetilde{\mathcal H} = \pi^*{\mathcal H}$. The map $\chi_P$ is injective if $C$ is not hyperelliptic.
\end{remark}

\begin{proposition}
The homomorphism 
$$(\eta_P)_x\,:\, (T_C)_x \otimes (\alpha)_x\, \longrightarrow\, {\mathcal H}_x \,\subset\, (V_P)_x
\,=\, H^0(C,\, K_C\otimes \alpha(2x))$$
varies holomorphically as $x$ moves over $C$. Thus it gives a holomorphic map 
\begin{equation}\label{etap}
\eta_P\ :\ T_C \otimes \alpha\ \longrightarrow \ {\mathcal H} \subset V_P.
\end{equation}
Moreover, the following diagram is commutative 
\begin{equation}\label{cdetaP}
\begin{tikzcd}[column sep=large] 
{} & H^1(C,\,\alpha) \otimes_{\mathbb C} {\mathcal O}_C\arrow{rd}{(i_P)_C} & \\
T_C\otimes \alpha \arrow{ru}{\psi_P} \arrow{rd}[swap]{-{\eta_P}} & & H^1(C,\, {\mathbb C}_{\alpha}) \\
& {\mathcal H}\arrow{ru}{\chi_P} &
\end{tikzcd}
\end{equation}
\end{proposition}

\begin{proof}
In Proposition \ref{jtilde} it was shown that there is an inclusion map of vector bundles 
$$
\widetilde{V}\ \stackrel{\overline{\epsilon}}{\hookrightarrow}\ {\mathcal N}\ \hookrightarrow\ \widetilde{W},
$$
and we have
\begin{equation}
\widetilde{j}_y\ =\ (j_{x} ,\, (\chi_P)_{x})\circ \overline{\epsilon}_y
\end{equation}
for all $y\, \in\, \widetilde{C}$ and $x\,=\, \pi(y)$, hence 
\begin{equation}
\widetilde{j}_y \circ \widetilde{\eta}_y\ =\ (j_{x} ,\, (\chi_P)_{x})\circ\overline{\epsilon}_y
 \circ \widetilde{\eta}_y.
\end{equation}
So using equations \eqref{sotto2}, \eqref{sopra} and \eqref{sotto1}, the following is obtained 
$$(j_{x} ,\, (\chi_P)_{x})\circ \overline{\epsilon}_y \circ \widetilde{\eta}_y
\ =\ (j_x, (\chi_P)_x) \circ (\eta_x, (\eta_P)_x) \circ \phi_y.$$
Since $j_{x}$ and $(\chi_P)_{x}$ are both injective if $C$ is not hyperelliptic, and
are injective on the complement of at most two points if $C$ is hyperelliptic (see \cite{cfg} and
Remark \ref{rem}), it follows that
\begin{equation}\label{somma}
\overline{\epsilon}_y \circ\widetilde{\eta}_y\ =\ (\eta_{x}, \,(\eta_P)_x) \circ \phi_y
\end{equation}
for all $y$ (recall that $x\,=\, \pi(y)$) when $C$ is not hyperelliptic, and when $C$ is
hyperelliptic, then \eqref{somma} holds for for all $y$ but finitely many points
of $\widetilde C$. 
Since $\overline{\epsilon}$, $\phi$ and $\widetilde{\eta}$ all vary holomorphically as the point
$y$ moves over $\widetilde C$, we conclude that $(\eta_P)_x $ varies holomorphically as the
point $x$ moves over $X$. Consequently, there is a homomorphism of vector bundles: 
$$\eta_P\ :\ T_C \otimes \alpha \ \longrightarrow\ {\mathcal H} \subset V_P.$$

The commutativity of the diagram in \eqref{cdetaP} is an immediate consequence
of the commutativity of the diagram in \eqref{cd}. 
\end{proof}

Since $E\,=\,p_* L$ (see \eqref{del}, \eqref{VP}), there is an isomorphism $H^0(C,\,E)\,\cong\, H^0(S,\,L)$
that associates to any $\beta \,\in\, H^0(C,\,E)$ the section $\widehat{\beta} \,\in\,
H^0(S,\,L)$ satisfying the condition
$$
\beta_x\, =\, \widehat{\beta}\big\vert_{\{x\}\times C}\, \in\, T_x^*C \otimes \alpha_x \otimes
H^0(C,\,K_C\otimes \alpha(2x))\, =\, E_x.
$$
It follows that for any $u\,\in\, T_xC \otimes \alpha_x$, and any $v\,\in\, T_y C \otimes \alpha_y$ with
$x\,\neq\, y$, we have $\beta_x (u)(v)\,=\,\widehat{\beta}(u,\,v)$. Thus there is a well-defined
section
\begin{equation}\label{ep}
\widehat{\eta}_P\ \,\in\ \, H^0(S,\,L),
\end{equation}
where $L$ is the line bundle in \eqref{del},
corresponding to $\eta_P$ in \eqref{etap}, and for any $u\,\in\,
T_xC \otimes \alpha_x$ and any $v\,\in\, T_y C \otimes \alpha_y$ with $x\,\neq\, y$, the equality $ (\eta_P)_x
(u)(v)\,=\,\widehat{\eta}_P(u,\, v)$ holds.

\section{A projective structure from the bidifferential}\label{se4}

As before, $\alpha$ is holomorphic line bundle on $C$ of order two equipped with a trivialization
of $\alpha^{\otimes 2}$.

It will be shown that the section $\widehat{\eta}_P\,\in\, H^0(S,\,K_S \otimes \alpha_S(2 \Delta))$
in \eqref{ep} induces a projective structure on $C$, which depends on $\alpha$.

\begin{lemma}\label{lem1}
The line bundle $(p^*\alpha)\otimes (q^*\alpha^*)$ on $S$ is canonically
trivial on an analytic open neighborhood of the diagonal $\Delta \,\subset\, C\times C \,=\, S$.
\end{lemma}

\begin{proof}
Let
\begin{equation}\label{e1}
\phi\ :\ \alpha^{\otimes 2}\ \longrightarrow\ {\mathcal O}_C
\end{equation}
be the trivialization of $\alpha^{\otimes 2}$.
Take any connected and simply connected open subset $U\,\subset\, C$. Choose
a holomorphic section $s_U \,\in\, H^0(U,\, \alpha\big\vert_U)$ such that
$\phi(s_U\otimes s_U)$ is the constant function $1$ on $U$, where $\phi$
is the isomorphism in \eqref{e1}. Let $$s^*_U\,\in\, H^0(U,\, \alpha^*\big\vert_U)$$
be the unique section such that $s^*_U (s_U) \,=\, 1$ (on $U$). Now we have
\begin{equation}\label{e2}
\gamma_U\ :=\ (p^*s_U)\otimes (q^*s^*_U)\ \in\ H^0(U\times U,\, \alpha_S\big\vert_{U\times U}).
\end{equation}

Note that $s_U$ in \eqref{e2} is not unique. In fact, there are exactly two choices of
$s_U$ that differ by only the sign (so $s_U$ and $-s_U$ are the two choices). Since
$(-s_U)^* \,=\, -(s^*_U)$ it follows immediately that $\gamma_U$ in \eqref{e2} is independent
of the choice of $s_U$.

Next it can be shown that $\gamma_U$ is independent of the choice of the isomorphism $\phi$ in
\eqref{e1}. To see this, replace $\phi$ by $c\cdot \phi$, where $c\,\in\, {\mathbb C}\setminus\{0\}$
is a nonzero scalar. Then $s_U$ changes to $d\cdot s_U$, where $d^2 \,=\, \frac{1}{c}$ (the two
solutions of $d$ correspond to the two choices of $s_U$). Also,
$s^*_U$ changes to $\frac{1}{d} s^*_U$. Thus $\gamma_U$ is independent of the choice of
the isomorphism $\phi$.

Covering $C$ by connected and simply connected open subsets $\{U_j\}_{j\in J}$ such that $U_i\cap
U_j$ is simply connected --- if it is nonempty --- we get a canonical trivialization of $p^*(\alpha)
\otimes q^*(\alpha^*)$ on an analytic neighborhood of the diagonal $\Delta \,\subset\, S$. 
\end{proof}

\begin{remark}
\label{triv}
Since $\alpha$ is isomorphic to its dual $\alpha^*$, Lemma \ref{lem1} implies that
$\alpha_S\,=\, (p^*\alpha)\otimes (q^*\alpha)$ is trivial on an analytic neighborhood of
$\Delta \,\subset\, C\times C \,=\, S$. But the trivialization of
$\alpha_S$ on an analytic neighborhood of $\Delta\,\subset\, S$
depends on the choice of the isomorphism of $\alpha^{\otimes 2}$
with ${\mathcal O}_C$.
\end{remark}

We recall the notation introduced in Section 2.1. For any $n\,\geq\, 1$, denote by $(n+1)\Delta$ the $n$-th 
infinitesimal neighborhood of the diagonal $\Delta \,\subset \,S$. The restriction 
$K_S(2\Delta)\big\vert_{\Delta}$ of $K_S(2\Delta)$ to $\Delta$ is the trivial bundle on $\Delta$ and the 
trivialisation $K_S(2\Delta)\big\vert_{\Delta}$ extends to a canonical trivialisation 
$K_S(2\Delta)\big\vert_{2\Delta}$ (see \cite[p.~754, Theorem 2.1]{BR1}, \cite[p.~688, Theorem 2.2]{BR2}). 
Denote by $s_0 \, \in\, H^0(2\Delta,\, K_S(2\Delta)\big\vert_{2\Delta})$ this canonical trivialisation.

We have the following theorem.

\begin{theorem}
\label{betaP}
The restriction of $\widehat{\eta}_P\,\in \,H^0(S,\,K_S\otimes \alpha_S(2 \Delta))$ (see
\eqref{ep}) to the second infinitesimal
neighbourhood of the diagonal gives a section of $H^0(3\Delta,\, K_{S}(2 \Delta)\big\vert_{3 \Delta})$
that induces a projective structure on the curve $C$, which depends on $\alpha$. 
\end{theorem}

\begin{proof}
By Remark \ref{triv} we know that $\alpha_S\,=\, (p^*\alpha)\otimes (q^* \alpha)$ is trivial on an analytic
neighborhood of the diagonal. Let $\gamma^\phi$
denote the trivialization, corresponding to $\phi$ in \eqref{e1}, of $\alpha_S$ on an analytic neighborhood of
$\Delta\,\subset\, S$ which has the property that the restriction of $\gamma^\phi$ to the reduced diagonal
$\Delta$ coincides with $\phi$ (see Remark \ref{triv}). This trivialization $\gamma^\phi$ produces
isomorphisms $$H^0(n\Delta,\, K_{S} \otimes \alpha_S (2 \Delta)\big\vert_{n \Delta}) \ \cong
\ H^0(n\Delta,\, K_{S}(2 \Delta)\big\vert_{n \Delta})$$ for all $n \,\geq\, 1$.

The restriction map $$H^0(S,\, K_{S} \otimes \alpha_S(2 \Delta))\, \longrightarrow\,
H^0(2\Delta,\, K_{S} \otimes \alpha_S (2 \Delta )\big\vert_{2 \Delta})
\,\cong\, H^0(2\Delta,\, K_{S}(2 \Delta)\big\vert_{2 \Delta}),$$ by construction, 
sends $\widehat{\eta}_P$ to the canonical trivialisation $s_0$.
Hence, under the restriction map
$$H^0(S,\, K_{S}\otimes \alpha_S(2 \Delta))\,\longrightarrow\, H^0(3\Delta,\, K_{S}\otimes
\alpha_S (2 \Delta )\big\vert_{3 \Delta}) \,\cong\, H^0(3\Delta,\, K_{S}(2 \Delta)\big\vert_{3 \Delta}),$$ 
the section $\widehat{\eta}_P$ is mapped to an element in ${\mathcal Q}(C)$ (see \eqref{QC}), and so
$\widehat{\eta}_P\big\vert_{3 \Delta}$
gives a projective structure on $C$ via the isomorphism \eqref{ephi}. This projective structure
of course depends on $\alpha$. 
\end{proof}

\section{The second fundamental form as a multiplication map}\label{se5}

It will be shown that the second fundamental form of the Prym map $\rho_P$ can be interpreted as a 
multiplication map by the section $\widehat{\eta}_P$ in \eqref{ep}.

We identify $H^0(C,\, K_C\otimes \alpha) \otimes H^0(C,\,K_C \otimes \alpha)$ with $H^0(S,\,
K_S\otimes \alpha_S)$; using this identification, $I_2(K_C \otimes \alpha)$
in \eqref{I2} becomes a subspace of $H^0(S,\, K_S\otimes \alpha_S( - \Delta))$. 
Since elements of the space $I_2(K_C \otimes \alpha)$ are symmetric, they in fact lie inside 
$H^0(S,\,K_S\otimes \alpha_S(-2\Delta))$. So for any $Q\,\in\, I_2(K_C\otimes\alpha)$, the product section 
$Q\cdot \widehat{\eta}_P$ lies inside $H^0(S,\,K_S^{\otimes 2})\,\cong\, H^0(C,\, K_C^{\otimes 2}) \otimes 
H^0(C, \, K_C^{\otimes 2})$.

Also, note that a trivialization of $\alpha^{\otimes 2}$ produces a trivialization of
$(\alpha_S)^{\otimes 2}$.

\begin{theorem}\label{prodotto}
With the above identifications, for $[(C,\, \alpha)]\, \in \,{\mathcal R}^0_g$ with $g\,\geq\, 7$,
the second fundamental form of the Prym map $$\rho_P\ :\ I_2(K_C \otimes \alpha)\ \longrightarrow\
{\rm Sym}^2(H^0(C,\,K_C^{\otimes 2}))$$ is the restriction to $I_2(K_C \otimes \alpha)$ of the multiplication map
$$
H^0(S,\, K_S \otimes \alpha_S(-2\Delta))\ \longrightarrow\ H^0(S,\,K_S^{\otimes 2} ),\ \ \ Q\
\longmapsto\ Q \cdot \widehat{\eta}_P .
$$
\end{theorem}

\begin{proof}
This follows from formula \eqref{rho_P}. If $x\,\neq\, y$, for
$$ (u,\,v )\ \in\ (T_C \otimes \alpha)_x \oplus (T_C \otimes \alpha)_y$$
recalling that the Schiffer variations in \eqref{Schiffer} are defined on 
$T_C^{\otimes 2}\, \cong\, (T_C \otimes \alpha) ^{\otimes 2}$, we have
\begin{gather*}
\rho_P(Q) (\xi_{u}\odot \xi_{v})\ =\ Q(u, v) (\eta_P)_x(u)(v)\\
=\ Q(u,v ) \widehat{\eta}_P(u, v)\ =\ (Q\cdot \widehat{\eta}_P) ( u,v )\ =\
(Q\cdot \widehat{\eta}_P) \, ( \xi_{u} \odot \xi_{v}).
\end{gather*}
This completes the proof.
\end{proof}

\begin{corollary}
\label{inj}
For any $[(C,\, \alpha)]\, \in \,{\mathcal R}^0_g$ with $g\,\geq\, 7$,
the second fundamental form of the Prym map $\rho_P\ :\ I_2(K_C \otimes \alpha)\ \longrightarrow\
{\rm Sym}^2(H^0(C,\,K_C^{\otimes 2}))$ 
is injective. 
\end{corollary}

\begin{proof}
By Theorem \ref{prodotto}, for any $Q \,\in\, I_2(K_C \otimes \alpha)\, \subset\,
H^0(S,\, K_S \otimes \alpha_S(-2\Delta))$, we have $\rho_P(Q)\, =\, Q \cdot \widehat{\eta}_P
\,\neq\, 0$, if $Q \,\neq\, 0$.
\end{proof}

Now it will be shown that, with the appropriate identifications, the section 
$\widehat{\eta}_P$ can be seen as the difference of the sections $\widetilde{\eta}$ and 
$\eta$, that govern the second fundamental forms of the Torelli maps for $\widetilde{C}$ and 
$C$ respectively.

Let
\begin{equation}\label{i2k}
I_2(K_{\widetilde{C}})\ \subset\ {\rm Sym}^2(H^0(\widetilde{C},\, K_{\widetilde C}^{\otimes 2}))
\end{equation}
be the kernel of the multiplication map ${\rm Sym}^2 (H^0(\widetilde{C},\, K_{\widetilde{C}})) \, \longrightarrow\,
H^0(\widetilde{C},\, K^{\otimes 2}_{\widetilde{C}})$.
Recall that the second fundamental form of the Torelli map for the curve $\widetilde{C}$ is the Hodge 
Gaussian map $\widetilde{\rho}\,:\, I_2(K_{\widetilde{C}})\, \longrightarrow\,
{\rm Sym}^2(H^0(\widetilde{C},\, K_{\widetilde C}^{\otimes 2}))$ given by the restriction to
$I_2(K_{\widetilde{C}})$ of the multiplication map
\begin{equation}\label{sff}
H^0(\widetilde{S},\, K_{\widetilde S} (-2\widetilde{\Delta}))\ \longrightarrow\ H^0(\widetilde{S},\,
K_{\widetilde S}^{\otimes 2} )\ \ \ Q\, \longmapsto\ Q \cdot\widetilde{\eta},
\end{equation}
where $\widetilde{\eta} \,\in\, H^0(\widetilde{S},\, K_{\widetilde S} (2\widetilde{\Delta}))$
(see \cite[Theorem 3.7]{cfg}).

Notice that 
$$H^0(\widetilde{C},\, K_{\widetilde C})\ =\ H^0(C,\, \pi_* K_{\widetilde C})\ =\ H^0(C,\,
\pi_* \pi^*K_{C})
$$
$$
=\ H^0(C,\, K_C\otimes\pi_*{\mathcal O}_{\widetilde C})\ =\ 
H^0(C,\, K_C\otimes ({\mathcal O}_C\oplus \alpha))\ =\ 
H^0(C,\, K_C) \oplus H^0(C,\, K_C \otimes \alpha)$$ and
$H^0(C, \,K_C) \,\cong \,H^0(\widetilde{C},\, K_{\widetilde C})^{+}$, while $H^0(C,\,
K_C \otimes \alpha) \,\cong\, H^0(\widetilde{C},\, K_{\widetilde C})^{-}$, where $H^0(\widetilde{C},\,
K_{\widetilde C})^{+}$ (respectively, $H^0(\widetilde{C},\, K_{\widetilde C})^{-}$) denotes the
invariant (respectively, anti-invariant) subspace of $H^0(\widetilde{C},\, K_{\widetilde C})$ under
the action of the nontrivial deck transformation $\sigma$ of the degree two covering map
$\pi\, :\, \widetilde{C}\, \longrightarrow\, C$. So we have
\begin{equation}\label{wh}
({\rm Sym}^2 (H^0(\widetilde{C},\, K_{\widetilde C})))^+\ \cong\ {\rm Sym}^2 (H^0(C,\, K_C)) \oplus
{\rm Sym}^2 (H^0(C,\, K_C \otimes \alpha)),
\end{equation}
$$
\text{and }\ \ \,H^0(\widetilde{C},\, K_{\widetilde C}^{\otimes 2})^+\ \cong\ H^0(C,\, K_C^{\otimes 2}).$$
The restriction of the multiplication map to the invariant part of the symmetric product 
\begin{equation}\label{es}
({\rm Sym}^2 (H^0(\widetilde{C},\, K_{\widetilde C})))^+\ \longrightarrow\ H^0(\widetilde{C},\,
K_{\widetilde C}^{\otimes 2})^+
\end{equation}
splits --- using \eqref{wh} --- as a sum of the two multiplication maps 
$$
{\rm Sym}^2 (H^0(C,\, K_C))\, \longrightarrow\, H^0(C,\, K_C^{\otimes 2}),\ \
{\rm Sym}^2 (H^0(C,\, K_C \otimes \alpha)) \, \longrightarrow\, H^0(C,\, K_C^{\otimes 2}).
$$
Let $I_2(K_{\widetilde C})^+\,\subset\, I_2(K_{\widetilde{C}})$ be the invariant part
for the action of $\sigma$ on $I_2(K_{\widetilde{C}})$ (see \eqref{i2k}). Using the above
decomposition of the map in \eqref{es} it follows that
$$I_2(K_{\widetilde C})^+\ \cong\ I_2(K_C) \oplus I_2(K_C \otimes \alpha),$$
where $I_2(K_C)$ is the kernel of the multiplication map
${\rm Sym}^2 (H^0(C,\, K_C)) \, \longrightarrow\, H^0(C,\, K^{\otimes 2}_C)$.

Observe that the restriction of $\widetilde{\rho}$ to $I_2(K_{\widetilde C})^+ $ is a map 
$$\widetilde{\rho}\big\vert_{I_2(K_{\widetilde C})^+}\ :\ I_2(K_{\widetilde C})^+\ \longrightarrow\
({\rm Sym}^2 (H^0(\widetilde{C},\, K_{\widetilde C}^{\otimes 2})))^+
$$
$$
\cong\ {\rm Sym}^2 (H^0(C,\, K_C^{\otimes 2})) \oplus {\rm Sym}^2 (H^0(C,\, K_C^{\otimes 2} \otimes \alpha)).$$ 
In \cite{cfgalois} it has been shown that 
$$p_1 \circ \widetilde{\rho}\big\vert_{I_2(K_C \otimes \alpha)}\ =\ \rho_P,$$
where $p_1\,:\, {\rm Sym}^2 (H^0(C,\, K_C^{\otimes 2})) \oplus {\rm Sym}^2 (H^0(C,\, K_C^{\otimes 2} \otimes \alpha))
\,\longrightarrow\, {\rm Sym}^2 (H^0(C, K_C^{\otimes 2}))$ is the projection onto the first direct summand.
Similarly, one can show that 
$$ p_1 \circ \widetilde{\rho}\big\vert_{I_2(K_C )}\ =\ \rho,$$
where $$\rho\ :\ I_2(K_C)\ \longrightarrow\ {\rm Sym}^2 (H^0(C,\, K_C^{\otimes 2}))$$
is the second fundamental form of the Torelli map for the curve $C$ (see \eqref{sff}),
which coincides with the multiplication by $\eta \,\in\, H^0(S,\, K_S(2\Delta))$ (\cite[Theorem 3.7]{cfg}). 

Recall that we have $\Phi^*({\mathcal O}_{S}({\Delta}))\, =\, {\mathcal O}_{\widetilde S}(\widetilde{\Delta}+
\Gamma)\, \supset\, {\mathcal O}_{\widetilde S}(\widetilde{\Delta})$, where $\Gamma$ is the graph of
$\sigma$. Moreover, we have $K_{\widetilde S}\, =\, \Phi^*(K_{S})$, and hence 
$$K_{\widetilde S}(2\widetilde{\Delta})\ \subset\ \Phi^*(K_{S}({2\Delta})),
$$
$$
\Phi_*K_{\widetilde S}(2\widetilde{\Delta})\ \subset\ \Phi_*\Phi^*(K_{S}({2\Delta}))
\ =\ K_S(2 \Delta) \otimes \Phi_*{\mathcal O}_{\widetilde S}. $$
Since $\Phi\,=\,(\pi,\, \pi)$, it follows that 
$$\Phi_*{\mathcal O}_{\widetilde S}\ \cong\ {\mathcal O}_{S}\oplus \alpha_S \oplus ({\mathcal O}_C
\boxtimes \alpha) \oplus (\alpha \boxtimes {\mathcal O}_C).$$
So there is a natural inclusion map
\begin{equation}\label{ek}
H^0(\widetilde{S},\, K_{\widetilde S} (2\widetilde{\Delta}))\, =\, H^0(S,\,
\Phi_*(K_{\widetilde S} (2\widetilde{\Delta}))\, \subset\, H^0(S,\, K_S(2 \Delta) )\oplus
H^0(S,\, K_S \otimes \alpha_S(2 \Delta) )
\end{equation}
$$\oplus H^0(S,\, K_C\boxtimes( K_C \otimes \alpha)(2 \Delta)) \oplus
H^0(S,\, (K_C\otimes \alpha) \boxtimes K_C(2 \Delta)).$$

\begin{theorem}\label{scomposizione}
Under the inclusion map in \eqref{ek},
$$\widetilde{\eta}\ \in\ H^0(S,\, K_S(2 \Delta) )\oplus H^0(S,\, K_S \otimes \alpha_S(2 \Delta) ),$$
and $\widetilde{\eta} \,=\, \eta + \widehat{\eta}_P$ (see \eqref{ep} for $\widehat{\eta}_P$). 
\end{theorem}

\begin{proof}
Using the inclusion map in \eqref{ek} we can write $\widetilde{\eta}\, =\, \eta_1 + \eta_2 + \eta_3
+ \eta_4$, where $\eta_i$ are the components of $\eta$ in the four direct summands.
Recall that $\widetilde{\eta}$ is invariant under the automorphism $k\,:=\, (\sigma,\, \sigma)$ of
$\widetilde{S}$ (see \cite[Lemma 3.6]{cfg}). Looking at the action of $k$ on the four summands
it follows immediately see that $k^*(\eta_1) \,=\, \eta_1$,\, $k^*(\eta_2)\, =\, \eta_2$,\,
$k^*(\eta_3) \,=\,
-\eta_3$,\, $k^*(\eta_4) \,=\, -\eta_4$, and since $k^*(\widetilde{\eta}) \,=\,
\widetilde{\eta}$, it is deduced that
$\eta_3 \,=\, 0$ and $\eta_4 \,=\, 0$. So $$\widetilde{\eta}\ =\ \eta_1 + \eta_2 \
\in\ H^0(S,\, K_S(2 \Delta) )\oplus H^0(S,\, K_S \otimes \alpha_S(2 \Delta)),$$
and from \eqref{somma} it follows
that $\eta_1 \,=\, \eta$, $\eta_2\, =\,\widehat{\eta}_P$. Hence 
\begin{equation}
\label{decomposizione}
\widetilde{\eta}\ =\ \eta + \widehat{\eta}_P.
\end{equation}
This completes the proof.
\end{proof}

\section{Relative version of the bidifferential}\label{se6}

Consider a holomorphic family of \'etale double covers. Such a family consists
of a holomorphic family $f\,:\, {\mathcal C}\,\longrightarrow\, B$ of compact Riemann surfaces of genus $g$,
parametrized by a complex manifold $B$, together with a holomorphic family of \'etale double covers
$\widetilde{\mathcal C}\, \longrightarrow\, \mathcal C$. In other words, $\widetilde{f}\,:\,
\widetilde{\mathcal C}\, \longrightarrow\, B$ is a holomorphic family of compact Riemann surfaces of genus
$2g-1$ equipped with an \'etale covering map $\varpi\, :\, \widetilde{\mathcal C}\, \longrightarrow\, \mathcal C$ 
of degree two such that $f\circ\varpi\,=\, \widetilde{f}$.

For any $b\, \in\, B$, the fiber $f^{-1}(b)$ (respectively, $\widetilde{f}^{-1}(b)$) will be denoted
by $C_b$ (respectively, $\widetilde{C}_b$). The direct image $\varpi_*{\mathcal O}_{\widetilde{\mathcal C}}
\,\longrightarrow\, \mathcal C$ decomposes as ${\mathcal O}_{\mathcal C}\oplus\alpha_{{\mathcal C}/B}$, where
$\alpha_{{\mathcal C}/B}$ is a nontrivial holomorphic line bundle of order two on $\mathcal C$. For any $b\, \in\,
B$, the holomorphic line bundle $\alpha_{{\mathcal C}/B}\big\vert_{C_b}\, \longrightarrow\, C_b$ will be
denoted by $\alpha_b$. Note that $\alpha_b$ is a nontrivial holomorphic line bundle of order two on $C_b$.

Define the fiber product
\begin{equation}\label{ecs}
{\mathcal S}\ :=\ {\mathcal C} \times_B {\mathcal C},
\end{equation}
and for $j\,=\, 1,\,2$, denote by $\widetilde{p_j}\,:\, {\mathcal S}\, \longrightarrow\, {\mathcal C}$
the natural projection to the $j$-th factor. Also, for $j\,=\, 1,\,2$, set
$F \, =\, f \circ \widetilde{p_j}\, :\, {\mathcal S}\, \longrightarrow\, B$, and denote
by $\Delta_B \, \subset\, {\mathcal S}$ the reduced relative diagonal. 

Consider the relative version of the map $\psi_P$ of \eqref{cdetaP}, that will continue to
be denoted by $\psi_P$: 
\begin{equation}\label{psirel}
0 \ \longrightarrow\ T_{{\mathcal C}/B} \otimes \alpha_{{\mathcal C}/B}\
\xrightarrow{\,\,\,\psi_P\,\,\,}\ R^1\widetilde{p}_{1*} ({\widetilde{p}_2}^* \alpha_{{\mathcal C}/B}).
\end{equation}
Consider the variation of Hodge structure for the family $\widetilde{p}_1\,:\, {\mathcal S}
\,\longrightarrow\, {\mathcal C}$, 
\begin{equation}
\label{vhs}
0\, \longrightarrow \,\widetilde{p}_{1*}((\widetilde{p}_2^*(K_{{\mathcal C}/B} \otimes \alpha_{{\mathcal C}/B}))
\,\longrightarrow\, R^1\widetilde{p}_{1*} (\widetilde{p}_2^* {\mathbb C}_{\alpha_{{\mathcal C}/B}})
\otimes {\mathcal O}_{\mathcal C}\,\longrightarrow\, R^1\widetilde{p}_{1*} (\widetilde{p}_2^*
\alpha_{{\mathcal C}/B}) \,\longrightarrow\, 0.
\end{equation}
Let 
\begin{equation}
\label{reliP}
i_P\ :\ R^1\widetilde{p}_{1*} (\widetilde{p}_2^* \alpha_{{\mathcal C}/B})\ \longrightarrow\
R^1\widetilde{p}_{1*} (\widetilde{p}_2^* {\mathbb C}_{\alpha_{{\mathcal C}/B}}) \otimes {\mathcal O}_{\mathcal C}
\end{equation}
be the ${\mathcal C}^{\infty}$ splitting given by the Hodge decomposition. 

Denote by ${\mathcal H}_{{\mathcal C}/B}\ \subset \ \widetilde{p}_{1*}((\widetilde{p}_2^*(K_{{\mathcal C}/B}
\otimes \alpha_{{\mathcal C}/B})(2 \Delta_B))$ the relative version of the vector bundle ${\mathcal H}$
of \eqref{cdetaP}, and by 
$$\eta_P\ :\ T_{{\mathcal C}/B} \otimes \alpha_{{\mathcal C}/B}\ \longrightarrow\
{\mathcal H}_{{\mathcal C}/B}\ \subset\ \widetilde{p}_{1*}((\widetilde{p}_2^*(K_{{\mathcal C}/B} \otimes
\alpha_{{\mathcal C}/B})(2 \Delta_B))$$
the relative version of $\eta_P$ (the same notation is used for it). 

We also have a relative version of the map $\chi_P$ in \eqref{j_P}
$$\chi_P\ :\ {\mathcal H}_{{\mathcal C}/B}\ \longrightarrow\ R^1\widetilde{p}_{1*} (\widetilde{p}_2^*
{\mathbb C}_{\alpha_{{\mathcal C}/B}}) \otimes {\mathcal O}_{\mathcal C},$$
and by \eqref{cdetaP} the following commutative diagram is obtained
\begin{equation}\label{cdetaPrel}
\begin{tikzcd}[column sep=large] 
{} &R^1 \widetilde{p}_{1*} (\widetilde{p}_2^* \alpha_{{\mathcal C}/B})\arrow{rd}{i_P} & \\
T_{{\mathcal C}/B} \otimes \alpha_{{\mathcal C}/B} \arrow{ru}{\psi_P} \arrow{rd}[swap]{-{\eta_P}} & &
R^1\widetilde{p}_{1*} (\widetilde{p}_2^* {\mathbb C}_{\alpha_{{\mathcal C}/B}}) \otimes {\mathcal O}_{\mathcal C} \\
& {\mathcal H}_{{\mathcal C}/B}\arrow{ru}{\chi_P} &
\end{tikzcd}
\end{equation}

By flat base change we have 
\begin{equation}
\label{fbc1}
R^1 \widetilde{p}_{1*} (\widetilde{p}_2^* {\mathbb C}_{\alpha_{{\mathcal C}/B}})\ \cong\
 f^*(R^1f_*{\mathbb C}_{\alpha_{{\mathcal C}/B}}),
\end{equation}
\begin{equation}
\label{fbc2}
R^1 \widetilde{p}_{1*} (\widetilde{p}_2^* {\alpha}_{{\mathcal C}/B})\ \cong\
f^*(R^1 f_* {\alpha}_{{\mathcal C}/B}),
\end{equation}
\begin{equation}
\label{fbc3}
\widetilde{p}_{1*}((\widetilde{p}_2^*(K_{{\mathcal C}/B} \otimes \alpha_{{\mathcal C}/B}))\ \cong\
f^*( f_* (K_{{\mathcal C}/B} \otimes \alpha_{{\mathcal C}/B}))\ =\ f^*({\mathcal F}^1),
\end{equation}
where ${\mathcal F}^1\ =\ f_* (K_{{\mathcal C}/B} \otimes \alpha_{{\mathcal C}/B})$ is the Hodge bundle that
fits in the exact sequence given by the variation of Hodge structure for $f$: 
\begin{equation}
\label{vhsf}
0\ \longrightarrow\ {\mathcal F}^1\ \longrightarrow\ R^1f_*{\mathbb C}_{\alpha_{{\mathcal C}/B}} \otimes
{\mathcal O}_{\mathcal C}\ \longrightarrow\ R^1 f_* {\alpha}_{{\mathcal C}/B}\ \longrightarrow\ 0;
\end{equation}
note that \eqref{vhs} is the pullback of \eqref{vhsf}. 
Considering \eqref{vhsf} and its pullback exact sequence via the map $\psi_P$, we get the commutative diagram 
\begin{equation}\label{hodge1}
\begin{tikzcd}
0\arrow{r}& f^* {\mathcal F}^1 \arrow{r} & f^*(R^1f_*{\mathbb C}_{{\alpha}_{{\mathcal C}/B}}) \otimes {\mathcal O}_{\mathcal C} \arrow{r}& f^*(R^1f_*{\alpha_{{\mathcal C}/B}}) \arrow{r} &0 \\
0\arrow{r}& f^* {\mathcal F}^1 \arrow{u}{=} \arrow{r}& {\mathcal H}_{{\mathcal C}/B} \arrow[hook]{u}\arrow{r} & T_{{\mathcal C}/B} \otimes \alpha_{{\mathcal C}/B} \arrow{u}{\psi_P}\arrow{r} \arrow{lu} {i_P\circ \psi_P= -\chi_P \circ \eta_P}& 0\\
\end{tikzcd}
\end{equation}

We have the following theorem.

\begin{theorem}
\label{c-infty-split}
The image of the above ${\mathcal C}^{\infty}$ homomorphism 
$$ i_P \circ \psi_P \,=\, i_P\circ \psi_P: T_{{\mathcal C}/B} \otimes \alpha_{{\mathcal C}/B}
\,\longrightarrow\, f^*(R^1f_*{\mathbb C}_{\alpha_{{\mathcal C}/B}}) \otimes
{\mathcal O}_{\mathcal C} \,\cong\, R^1\widetilde{p}_{1*}
(\widetilde{p}_2^* {\mathbb C}_{\alpha_{{\mathcal C}/B}})\otimes {\mathcal O}_{\mathcal C} $$
lies in the subsheaf ${\mathcal H}_{{\mathcal C}/B}\, \subset\,R^1\widetilde{p}_{1*}
(\widetilde{p}_2^* {\mathbb C}_{\alpha_{{\mathcal C}/B}})\otimes {\mathcal O}_{\mathcal C}$,
and it gives a ${\mathcal C}^{\infty}$ splitting of the bottom exact sequence of
the diagram in \eqref{hodge1}. 
\end{theorem}

\begin{proof}
The theorem follows immediately from diagram \eqref{cdetaPrel}.
\end{proof}

\section{A canonical section of the bundle of projective structures}\label{se7}

Denote by
$${\mathcal L}\ :=\ \widetilde{p}_1^*(K_{{\mathcal C}/B} \otimes
\alpha_{{\mathcal C}/B}) \otimes \widetilde{p}_2^*(K_{{\mathcal C}/B}
\otimes \alpha_{{\mathcal C}/B}) (2 \Delta_B),$$ by
$${\mathcal V}:= F_*({\mathcal L}/{\mathcal L}(-3\Delta_B)), \ \ {\mathcal V}_2:= F_*({\mathcal L}/{\mathcal L}(-2\Delta_B)).$$
So, for all $b \in B$, denoting by $S_b\,:=\, C_b \times C_b$, and by
$\Delta_b \subset S_b$ the diagonal, we have
$${\mathcal V}_b\, =\, H^0(3 \Delta_b,\, (K_{S_b} \otimes \alpha_{S_b}(2 \Delta_b))_{|3 \Delta_b})
\,\cong\, H^0(3 \Delta_b,\, K_{S_b}(2 \Delta_b)_{|3 \Delta_b}),$$
$$({\mathcal V}_2)_b\, =\, H^0(2 \Delta_b,\, (K_{S_b} \otimes \alpha_{S_b}
(2 \Delta_b))_{|2 \Delta_b})\,\cong\, H^0(2 \Delta_b,\, K_{S_b}(2 \Delta_b)_{|2 \Delta_b}).$$

There is a natural surjective homomorphism
$$G\ :\ {\mathcal V}\ \longrightarrow\ {\mathcal V}_2$$
such that for every $b \,\in\, B$,
$$G_b\ :\ H^0(3\Delta_b,\, K_{S_b}(2 \Delta_b)\big\vert_{3\Delta_b})\ \longrightarrow\
H^0(2\Delta_b,\, K_{S_b}(2 \Delta_b)\big\vert_{2\Delta_b})$$ is the
restriction map. The vector bundle ${\mathcal V}_2$ has a
tautological holomorphic section $\sigma_0$ given by the canonical
trivialization of $K_{C \times C}(2 \Delta) \vert_{2\Delta}$ for any
curve $C$. Denote
\begin{equation}\label{vhat}
\widehat{\mathcal V}\ :=\ G^{-1}({\sigma}_0)\ \subset\ {\mathcal V}.
\end{equation}
Note that $\widehat{\mathcal V}$ is a holomorphic affine bundle over
$B$ modelled on the vector bundle
$f_*({K_{{\mathcal C}/B}}^{\otimes 2})$. Indeed, this follows
immediately from the fact that
$Ker (G) = f_*({K_{{\mathcal C}/B}}^{\otimes 2})$.

We have the following Lemma of \cite{bcfp}.

\begin{lemma}[{\cite[Lemma 2.1]{bcfp}}]\label{lem1a}
The $C^\infty$ (respectively, holomorphic) sections of the fiber
bundle $\widehat{\mathcal V} \, \longrightarrow\,B$ are in a natural
bijective correspondence with the $C^\infty$ (respectively,
holomorphic) families of projective structures for the family of
curves $\mathcal C$.
\end{lemma}

\begin{remark}\label{beta}
The assignment, for any $b \in B$ of the restriction of the
meromorphic form
$({\eta}_P)_b \in H^0(K_{S_b} \otimes \alpha_{S_b}(2 \Delta_b))$ to
$3\Delta_b$, gives a ${\mathcal C}^{\infty}$ section of
$\widehat{\mathcal V}$ that we denote by $\beta_P$ (see Theorem \ref{betaP}).
\end{remark}

Denote by $\overline{\partial}_{\mathcal V}$ the Dolbeault operator
for the holomorphic vector bundle $\mathcal V$. Given a $C^\infty$
section $ \beta\, :\, B\, \longrightarrow\, \widehat{\mathcal V}, $
since $\beta$ is also a section of $\mathcal V$, we have
$ \overline{\partial}_{\mathcal V}(\beta)\, \in\, \Omega^{0,1}(B,
\,{\mathcal V})$. The projection $G$ is holomorphic and
$G(\beta)$ is a holomorphic section of ${\mathcal V}_2$. Hence
$G\circ\overline{\partial}_{\mathcal V}(\beta)\, \in\, \Omega^{0,1}(B,
\,{\mathcal V}_2)$, and it follows that
$$
G\circ\overline{\partial}_{\mathcal V}(\beta)\ =\
\overline{\partial}_{{\mathcal V}_2} (G(\beta))\ =\ 0,
$$
where $\overline{\partial}_{{\mathcal V}_2}$ is the Dolbeault operator
for the holomorphic vector bundle ${\mathcal V}_2$. So we have
$ \overline{\partial}_{\mathcal V}(\beta)\, \in\, \Omega^{0,1}(B, \,
f_*({K_{{\mathcal C}/B}}^{\otimes 2}))$. Let
\begin{equation}\label{ega}
M\ :\ B\ \longrightarrow\ {\mathcal R}_g
\end{equation}
be the holomorphic map to the moduli space of \'etale double covers of
curves of genus $g$, corresponding to the above family
$f:\mathcal C \, \longrightarrow\, B$.

Assume that we have two $C^\infty$ sections of
$\widehat{\mathcal V}\, \longrightarrow\,B$ and
$\beta_1,\, \beta_2\, :\, B\,\longrightarrow\,\widehat{\mathcal V}$
such that
$$\overline{\partial}_{\mathcal V}(\beta_1)\ =\
\overline{\partial}_{\mathcal V}(\beta_2).$$ Since
$G(\beta_1-\beta_2)\,=\, 0$, it is deduced that
$$ \beta_1-\beta_2\ \in\ H^0(B,\, M^* \Omega^{1,0}_{{\mathcal
R}_g}),$$ where $M$ is the modular map in \eqref{ega}.

In \cite[Theorem 4.1]{fnpt} Favale, Naranjo, Pirola and Torelli have
shown that for $g \,\geq\, 5$,
$H^0({\mathcal R}_g,\, \Omega^1_{{\mathcal R}_g}\,) =\,0$, hence we have
the following:

\begin{corollary} 
If $B\, =\, {\mathcal R}_g$ and
$\overline{\partial}_{\mathcal V}(\beta_1)\,
=\,\overline{\partial}_{\mathcal V}(\beta_2)$, then
$\beta_1 \,=\, \beta_2$.
\end{corollary}

Assume $B = {\mathcal R}_g$ and consider the natural $(1,\,1)$-form
$\overline{\partial}(\beta_P)$ on ${\mathcal R}_g$, where $\beta_P$ is
the ${\mathcal C}^{\infty}$-section of $\widehat{\mathcal V}$ defined
in Remark \ref{beta}.

\begin{theorem}\label{debar}
The $(1,1)$-form $\overline{\partial}(\beta_P)$ is a nonzero scalar
multiple of ${\mathcal P}^* (\omega_S)$, where $\omega_S$ is the
K\"ahler form of the Siegel metric on ${\mathcal A}_{g-1}$ and
${\mathcal P}$ denotes the Prym map \eqref{prym}.
\end{theorem}

\begin{proof}
First we recall some results from \cite[section 4]{bcfp}. Consider
the universal family
$$u\ :\ {\mathcal U}\ \longrightarrow\ {\mathcal A}_{g-1}$$ of
principally polarized abelian varieties, and the exact sequence
coming from the Hodge filtration:
\begin{equation}\label{hodgeab}
0\, \longrightarrow\, {\mathcal F}^1\, \longrightarrow\, R^1 u_*{\mathbb C}_{\mathcal U}
\otimes {\mathcal O}_{ {\mathcal A}_{g-1}}\,\longrightarrow\, R^1 u_*{\OO}_{\mathcal U}
\, \cong\, ({\mathcal F}^1)^{\vee}\, \longrightarrow\, 0,
\end{equation}
where we continue to denote by ${\mathcal F}^1$ the Hodge
bundle. Notice that ${\mathcal F}^1$ in \eqref{vhsf} is the pullback
of ${\mathcal F}^1$ in \eqref{hodgeab} via the Prym map
${\mathcal P}\,:\, {\mathcal R}_g \,\longrightarrow\, {\mathcal A}_{g-1}.$

Denote by
$$\iota\,:\, R^1 u_*{\OO}_{\mathcal U}\, \cong\, ({\mathcal
F}^1)^{\vee}\, \longrightarrow\, R^1 u_*{\mathbb C}_{\mathcal U}
\otimes {\mathcal O}_{ {\mathcal A}_{g-1}} $$ the
${\mathcal C}^{\infty}$ splitting of \eqref{hodgeab} given by the
Hodge decomposition.

Tensoring \eqref{hodgeab} by ${\mathcal F}^1$, the
following short exact sequence is obtained:
\begin{equation}\label{hodgeab1}
0 \,\longrightarrow\, {\mathcal F}^1 \otimes {\mathcal F}^1 \, \longrightarrow\,
R^1 u_*{\mathbb C}_{\mathcal U} \otimes {\mathcal F}^1 \,\longrightarrow\,
({\mathcal F}^1)^{\vee} \otimes {\mathcal F}^1 \, \longrightarrow\, 0,
\end{equation}
which has a ${\mathcal C}^{\infty}$ splitting given by
$\iota'\,:=\, \iota \otimes Id_{{\mathcal F}^1}\,:\, ({\mathcal
F}^1)^{\vee}\otimes {\mathcal F}^1\,\longrightarrow \, R^1
u_*{\mathbb C}_{\mathcal U} \otimes {\mathcal F}^1 $. Consider the
map
$$s\ :\ \OO_{{\mathcal A}_{g-1}}\ \longrightarrow\ ({\mathcal F}^1)^{\vee}
\otimes {\mathcal F}^1,\ \ c\ \longmapsto\ c \cdot Id,$$ and take the
pullback of \eqref{hodgeab1} via $s$:
\begin{equation}\label{hodgeab2}
\begin{tikzcd}
0\arrow{r}& {\mathcal F}^1\otimes {\mathcal F}^1\arrow{r} & R^1u_*{\mathbb C}_{{\mathcal U}} \otimes {\mathcal F}^1 \arrow{r}& ({\mathcal F}^1)^{\vee} \otimes {\mathcal F}^1\arrow{r} &0 \\
0\arrow{r}& {\mathcal F}^1\otimes {\mathcal F}^1\arrow{u}{=} \arrow{r}& {\mathcal G} \arrow[hook]{u}\arrow{r} & {\mathcal O}_{{\mathcal A}_{g-1}} \arrow{u}{s}\arrow{r} \arrow{lu}{\iota' \circ s} & 0\\
\end{tikzcd}
\end{equation}
Since the image of $\iota' \circ s$ lies in ${\mathcal G}$, it gives a
${\mathcal C}^{\infty}$ splitting of the bottom exact sequence. Taking
the quotient of the bottom exact sequence in \eqref{hodgeab2} by
$\Lambda^2 {\mathcal F}^1$, we get the short exact sequence
\begin{equation}\label{l}
0\, \longrightarrow\, {\rm Sym}^2({\mathcal F}^1)\, \longrightarrow\, {\mathcal G}^+\,:=\,
{\mathcal G}/\Lambda^2 {\mathcal F}^1 \,\stackrel{l}{\longrightarrow}\,
{\mathcal O}_{{\mathcal A}_{g-1}} \, \longrightarrow\, 0
\end{equation}
that admits a ${\mathcal C}^{\infty}$ splitting
$\sigma\,:=\,t \circ \iota' \circ s \,:\, {\mathcal O}_{{\mathcal
A}_{g-1}}\, \longrightarrow\, {\mathcal G}^+$, where
$t \,:\, {\mathcal G} \,\longrightarrow\, {\mathcal G}^+$ is the
projection. Since the morphism $l$ in \eqref{l} is holomorphic ---
denoting $h\,:=\, \sigma(1)$ -- it follows that
$$l(\overline{\partial}h)\ =\ \overline{\partial}(l \circ h)\ =\
\overline{\partial}(1) \ =\ 0.$$ So the form
\begin{equation}\label{omega}
\omega\ :=\ \overline{\partial}h
\end{equation}
is a $(0,\,1)$-form on ${\mathcal A}_{g-1}$ with values in
${\rm Sym}^2({\mathcal F}^1) \,=\, T^*{{\mathcal A}_{g-1}}$, hence it
is a $(1,1)$-form on ${\mathcal A}_{g-1}$. In \cite[Proposition
4.3]{bcfp} is is proven that the $(1,\,1)$-form $\omega$ is a non-zero
scalar multiple of the K\"ahler form $\omega_S$ for the Siegel metric
on ${\mathcal A}_{g-1}$.

Now, take the tensor product of the exact sequence \eqref{hodge1} by
${\mathcal K}_{\alpha}:= K_{{\mathcal C}/{{\mathcal R}_g}} \otimes
\alpha_{{\mathcal C}/{{\mathcal R}_g}}$. So we get the commutative
diagram
\begin{equation}\label{hodge2}
\begin{tikzcd}
0\arrow{r}& f^* {\mathcal F}^1 \otimes {\mathcal K}_{\alpha}\arrow{r} & f^*(R^1f_*{\mathbb C}_{{\alpha}_{{\mathcal C}/{{\mathcal R}_g}}})\otimes {\mathcal K}_{\alpha}\arrow{r}& f^*(R^1f_*{\alpha_{{\mathcal C}/{{\mathcal R}_g}}})\otimes {\mathcal K}_{\alpha}\arrow{r} &0 \\
0\arrow{r}& f^* {\mathcal F}^1\otimes {\mathcal K}_{\alpha}\arrow{u}{=} \arrow{r}& {\mathcal H}_{{\mathcal C}/{{\mathcal R}_g}} \otimes {\mathcal K}_{\alpha} \arrow[hook]{u}\arrow{r} & {\mathcal O}_{\mathcal C} \arrow{u}{\psi_P \otimes Id_{{\mathcal K}_{\alpha}}}\arrow{r} \arrow{lu} & 0\\
\end{tikzcd}
\end{equation}
with a ${\mathcal C}^{\infty}$ splitting of the bottom exact sequence
given by
$(i_P \circ \psi_P) \otimes Id_{{\mathcal K}_{\alpha}}\, =\, -
(\chi_P\circ \eta_P) \otimes Id_{{\mathcal K}_{\alpha}}$.

Pushing-forward \eqref{hodge2} via $f_*$, the commutative
diagram is obtained:
\begin{equation}\label{hodge3}
\begin{tikzcd}
0\arrow{r}& {\mathcal F}^1\otimes {\mathcal F}^1\arrow{r} & R^1f_*{\mathbb C}_{{\alpha}_{{\mathcal C}/{{\mathcal R}_g}}} \otimes {\mathcal F}^1 \arrow{r}& ({\mathcal F}^1)^{\vee} \otimes {\mathcal F}^1\arrow{r} &0 \\
0\arrow{r}& {\mathcal F}^1\otimes {\mathcal F}^1\arrow{u}{=} \arrow{r}& f_*({\mathcal H}_{{\mathcal C}/{{\mathcal R}_g}} \otimes {\mathcal K}_{\alpha}) \arrow[hook]{u}\arrow{r} & {\mathcal O}_{{\mathcal R}_g} \arrow{u}\arrow{r} \arrow{lu} & 0\\
\end{tikzcd}
\end{equation}
where the right vertical map sends $c$ to
$c \cdot Id_{{\mathcal F}^1}$.

Now, taking the quotient in the bottom exact sequence \eqref{hodge3}
by $\Lambda^2{\mathcal F}^1$, the following short exact sequence is obtained
\begin{equation}\label{hodge5}
0\, \longrightarrow\, {\rm Sym}^2({\mathcal F}^1)\, \longrightarrow\,
f_*({\mathcal H}_{{\mathcal C}/{{\mathcal R}_g}} \otimes {\mathcal K}_{\alpha})^+
\,\longrightarrow\, {\mathcal O}_{{\mathcal R}_g} \,\longrightarrow\, 0,
\end{equation}
where
$f_*({\mathcal H}_{{\mathcal C}/{{\mathcal R}_g}} \otimes {\mathcal
K}_{\alpha})^+ \,:=\, f_*({\mathcal H}_{{\mathcal C}/{{\mathcal
R}_g}} \otimes {\mathcal K}_{\alpha})/ \Lambda^2{\mathcal F}^1$,
that has a ${\mathcal C}^{\infty}$ splitting given by $\eta_P$.

Recall that
${\mathcal H}_{{\mathcal C}/{{\mathcal R}_g}}\,\subset\,\widetilde{p}_{1*}(\widetilde{p}_{2}^*
({\mathcal K}_{\alpha})(2 \Delta_B))$, so
$$f_*({\mathcal H}_{{\mathcal C}/{{\mathcal R}_g}} \otimes {\mathcal K}_{\alpha})\subset f_*(
\widetilde{p}_{1*}(\widetilde{p}_2^*({\mathcal K}_{\alpha})(2
\Delta_{{\mathcal R}_g})) \otimes {\mathcal K}_{\alpha})\,\cong$$
$$ \cong\, f_* (\widetilde{p}_{1*}(\widetilde{p}_1^*({\mathcal K}_{\alpha}) \otimes
\widetilde{p}_2^*({\mathcal K}_{\alpha}) (2 \Delta_{{\mathcal R}_g})))\,=$$
$$=\, F_*(\widetilde{p}_1^*(K_{{\mathcal C}/{{\mathcal R}_g}} \otimes
\alpha_{{\mathcal C}/{{\mathcal R}_g}})\otimes
\widetilde{p}_2^*(K_{{\mathcal C}/{{\mathcal R}_g}} \otimes
\alpha_{{\mathcal C}/{{\mathcal R}_g}}) (2 \Delta_{{\mathcal R}_g}))\,
=\, F_*({\mathcal L}).$$ Taking the push-forward of \eqref{hodge5} via
the codifferential of the Prym map
${\mathcal P}\,:\, {\mathcal R}_g \,\longrightarrow\, {\mathcal A}_{g-1}$ we
obtain the commutative diagram
\begin{equation}\label{hodge4}
\begin{tikzcd}
0\arrow{r}& {\rm Sym}^2({\mathcal F}^1)\arrow{r}\arrow{d}{(d{\mathcal P})^*} & f_*({\mathcal H}_{{\mathcal C}/{{\mathcal R}_g}} \otimes {\mathcal K}_{\alpha})^+ \arrow{r}\arrow{d}{r}& {\mathcal O}_{{\mathcal R}_g} \arrow{r} \arrow{d}{=}&0 \\
0\arrow{r}& T^*{{\mathcal R}_g} \arrow{r}& {\mathcal V}' \arrow{r} & {\mathcal O}_{{\mathcal R}_g} \arrow{r} & 0\\
\end{tikzcd}
\end{equation}
where $r$ is the restriction to the sub-bundle
$f_*({\mathcal H}_{{\mathcal C}/{{\mathcal R}_g}} \otimes {\mathcal
K}_{\alpha})^+ \subset F_*({\mathcal L})/\Lambda^2 {\mathcal F}^1$
of the restriction map to the second order infinitesimal neighborhood
of $\Delta_{{\mathcal R}_g}$, so $ {\mathcal V}'\subset {\mathcal V}$.

The bottom exact sequence admits a ${\mathcal C}^{\infty}$ splitting
given by $\beta_P$, and, by construction, the $(1,\,1)$-form on
${\mathcal R}_g$ given by $ {\overline{\partial}}(\beta_P)$ coincides
with ${\mathcal P}^*(\omega)$, where $\omega$ is the form in
\eqref{omega}. Hence $ {\overline{\partial}}\beta_P$ is a non-zero
scalar multiple of ${\mathcal P}^* (\omega_S)$.
\end{proof}

We now compare the section $\beta_P$ with other known sections of
projective structures.

\begin{lemma}\label{lems}
Let $\omega$ be an orbifold K\"ahler form on ${\mathcal
A}_{g-1}$. Its pullback ${\mathcal P}^*(\omega)$ on
${\mathcal R}_g$ via the Prym map is not equal to the pullback from
of a $(1,\,1)$-form on ${\mathcal M}_{g}$.
\end{lemma}

\begin{proof}
Denote by
$f\,:\, {\mathcal R}_{g}\,\longrightarrow\, {\mathcal M}_g$ the
forgetful map. It is a degree $2^{2g} -1$ \'etale covering. Take a
$(1,\,1)$-form $\gamma$ on $ {\mathcal M}_g$. Take a point
$[(C, \alpha)]\,\in\, {\mathcal R}_{g}$, where the differential of
the Prym map ${\mathcal P}$ is not injective. If
${\mathcal P}^*(\omega)\, = \,f^*(\gamma)$, then for every
$x \,\in\, f^{-1}([C])$, the form $f^*(\gamma)$ is degenerate.

On the other hand, in \cite[Proposition 5.4.2]{debarre_duke} it is
shown that over the generic point of the bielliptic locus there
exist points where the differential of the Prym map is injective,
and other points where it is not. Hence, ${\mathcal P}^*(\omega)$ is
nondegenerate in certain points of a fiber over a general
bielliptic curve, and degenerate on other points in the same fiber,
so it cannot be a pullback $f^*(\gamma)$.
\end{proof}

\begin{theorem}\label{different}
The section $\beta_P$ of $\widehat{\mathcal V}$ is not the pullback
of a section of the bundle of projective structures on
${\mathcal M}_g$. In particular, it is not the pullback of the
section $U$ given by uniformisation, nor of the sections constructed
in \cite{bcfp}, \cite{bgv1}, \cite{bv}.
\end{theorem}

\begin{proof}
Let $\beta $ be a section of projective structures on $\mathcal M_g$
and let again
$f\,:\, {\mathcal R}_{g}\,\longrightarrow\, {\mathcal M}_g$ the
forgetful map. By Theorem \ref{debar}
$\debar \beta_P \,=\, \lambda \mathcal P^*(\omega_S)$ for some
constant $\lambda$. Since $\debar \beta$ is a $(1,\,1)$-form on
$\mathcal{M}_g$ and $\omega_S$ is a K\"ahler metric on
$\mathcal{A}_{g-1}$, the previous lemma implies that
$f^*(\debar \beta) \,\neq\, \lambda {\mathcal P}^*(\omega_S)$. But
$\debar (f^*\beta)\, =\, f^*( \debar \beta)$, so
$\debar (f^*\beta) \,\neq\, \debar \beta_P$ and
$f^*\beta \,\neq\, \beta_P$.
\end{proof}

We remark that Lemma \ref{lems} gives a short proof of the fact that
the section $\beta^P$ constructed in \cite{bgv} is not the pull-back
of a section over ${\mathcal M}_g$, a fact proved in \cite{bgv} by a
different argument.

Finally we compare $\beta_P$ and $\beta^P$.

\begin{lemma}\label{Whitney}
Let $M$ be a complex orbifold. Let $X \,\subset\, M$ be a complex
analytic subset and let $Y\,\subset\, M$ be a complex
submanifold. Let $\gamma $ be a smooth $(1,\,1)$-form on $M$. If
$Y\subset X$, and $\gamma\big\vert_{X_\mathrm{reg}}\,\equiv\, 0$, then
$\gamma\big\vert_Y \,\equiv\, 0$.
\end{lemma}

\begin{proof}
The problem is local. So we can assume that $M$ is a domain in
$\mathbb C^n$. First some definitions and some theorems of
Whitney \cite{whit1,whit2} are recalled. If $V$ is a (complex) analytic subset of
$M$ and $p\in V$, Whitney defined several cones attached to $V$ at
$p$. We need only two of them, namely $C_3(V,p)$, which is also
denoted simply by $C(V,p)$, and $C_4(V,p)$. Just the
following four facts are needed:
\begin{enumerate}
\item if $V$ is smooth at $p$, then $C(V,p)\, =\, T_pV$;

\item if $p\,\in \,W \,\subset\, V$, then $C(W,p)\,\subset\, C(V,p)$;

\item by definition $C_4(V,p)$ is the set of $v \,\in\, \mathbb C^n$
such that there exist sequences $\{p_j\}$ and $\{v_j\}$ such that
$p_j \,\in\, V_{\mathrm{reg}}$, $v_j \,\in\, T_{p_j}V$,
$p_j \,\to\, p$ and $v_j \,\to\, v$ (in the topology of
$\mathbb C^n$);

\item $C(V,p) \, \subset\, C_4(V,p)$.
\end{enumerate}
(See \cite[pp. 211-223]{whit1} for the statements and \cite{whit2}
for the proofs; other nice references are \cite[Ch. 7]{whit3} and
\cite[\S 9]{chirka}.) Define
$q\, :\, T^{1,0}M\,\longrightarrow\, \mathbb C$ by
$q(v)\,:=\gamma\,(v,\,\overline{v})$. Then $q$ determines $\gamma$,
just as the norm determines the scalar product in a Hilbert space.
By assumption $q|_{TX_{\mathrm{reg}}}\equiv 0$. We wish to prove
that $q|_{TY}\, \equiv\, 0$. Take $y \,\in\, Y$ and
$v \,\in\, T_yY$. Then $v \,\in\, C(Y,y) \,=\, T_yY$, so
$v \,\in\, C(X,y)$ and therefore $v\,\in \,C_4(X,y)$. Hence there
are $x_j\,\in\, X_{\mathrm{reg}}$ and
$v_j \,\in\, T_{x_j} X_{\mathrm{reg}}$ with $x_j \,\to\, y$ and
$v_j \,\to\, v$. By assumption $q(v_j)\, =\,0$, so
$q(v)\, =\, \lim q(v_j)\,=\, 0$. This proves that $q\,\equiv\, 0$
on $TY$ and hence that $\gamma |_Y\,\equiv\, 0 $, as desired.
\end{proof}

\begin{theorem}
The ${\mathcal C}^{\infty}$ section $\beta_P$ of
$\widehat{\mathcal V}$ differs from the section $\beta^P$
constructed in \cite{bgv}
\end{theorem}

\begin{proof}
We know from \cite[Theorem 1.4]{bgv} that
$\debar \beta^P \,=\,{\mathcal P}^* ( 8 \pi \Theta^*\omega_{FS})$.
{}From Theorem \ref{debar} above we know also that
$\debar \beta_P \,=\, \lambda {\mathcal P}^* \omega_S$. Set for simplicity
$\gamma\,:=\,\lambda \omega_S - 8 \pi \Theta^*\omega_{FS}$. This is
a form on $\mathcal A_{g-1}$. If we prove that
${\mathcal P}^*(\gamma) \,\neq\, 0$ on ${\mathcal R}_g$, then
$$\debar (\beta_P - \beta^P ) \ \neq\ 0,$$ and hence
$\beta_P \,\neq\, \beta^P$.

So, let
$j\,:\, {\mathcal M}_{g-1}\,\longrightarrow\, {\mathcal A}_{g-1}$ be
the Torelli map. Let $X$ be the closure of $\im {\mathcal P}$ in
$\mathcal A_{g-1}$. Let $\Omega \subset\mathcal M_{g-1,\reg}$ be
open and dense (for the complex topology) and such that $j$ is an
embedding on $\Omega$. (For example, one can take $\Omega $ to be the locus
in $\mathcal M_{g-1}$ corresponding to curves without nontrivial
automorphisms if $g\,>\, 3$ and $\Omega = \mathcal M_{2,\mathrm{reg}}$
if $g\,=\,3$.) Then $Y\,:=\,j(\Omega)$ is a complex submanifold contained
in $X$, since the Jacobian locus is contained in the closure of the
Prym locus (see \cite{wirtinger}, \cite{beauville}, \cite{donagi-smith}). 

Assume by contradiction that ${\mathcal P}^*\gamma \,=\,0$. The
restriction
${\mathcal P}\big\vert_{\mathcal P^{-1}(X_{\mathrm{reg}})} \,:\,
{\mathcal P}^{-1}(X_{\mathrm{reg}}) \,\longrightarrow\,
X_{\mathrm{reg}}$ is generically an isomorphism. So
${\mathcal P}^*\gamma =0$ implies that $\gamma|_{X_\reg} \,=\,0$.
By Lemma \ref{Whitney} also $\gamma|_{Y} = 0$. Since
$j\,:\, \Omega\, \longrightarrow\, Y$ is an isomorphism, we deduce
that $j^*\gamma\, =\,0$, i.e.,
$j^* \omega_S \,=\, 8\pi j^* \Theta^* \omega_{FS}$. But this
contradicts Theorem 5.4 in \cite{bv}. Indeed that paper studies two
sections $\beta^\theta$ and $\beta^\eta$. Theorem 5.4 in \cite{bv}
proves that $\debar \beta^\theta \,\neq\, \debar \beta^\eta$, while
$\debar \beta^\theta \,=\, 8 \pi j^* \Theta^* \omega_{FS} $ by (4.16) in
that paper and $\debar \beta^\eta\, =\, \pi j^* \omega_{S} $ by
(4.18) in that paper. So we reach a contradiction. This proves that
$\mathcal P^*\gamma \,\neq\, 0$, and hence
$\beta_P \,\neq\, \beta^P$.
\end{proof}

\section*{Declarations}

No AI was used in any form.

No data were generated or used.

None of the authors has any conflict of interests to declare.

%%%%%%%%%%%%%%%%%%%%%%%%%%%%%%%%%%%%%%%%%%%%%%%%%%%%%%%%%%%%%%%%%%%%%%%%%%%%%%%%%%%%%%%%%%%%%%%%

\end{document}